\documentclass[a4paper,11pt]{article}
\usepackage{amsmath,amsthm}
\usepackage{authblk}
\usepackage{aligned-overset}
\usepackage[style=numeric-comp,maxbibnames=99,bibencoding=utf8,firstinits=true]{biblatex}
\usepackage{booktabs}
\usepackage{cancel}
\usepackage[font=small]{caption}
\usepackage[hmargin={26mm,26mm},vmargin={30mm,35mm}]{geometry}
\usepackage{nicefrac}
\usepackage{paralist}
\usepackage[colorlinks, allcolors=blue]{hyperref}
\usepackage{subcaption}
\usepackage[compact,small]{titlesec}
\usepackage{tikz}
\usepackage{tikz-cd}
\usepackage{pgfplotstable}
\usepackage{newtxmath}
\usepackage{newtxtext}
\usepackage{enumitem}
\usepackage{bm}
\usepackage{algpseudocode}
\usepackage{algorithm}
\usepackage{thmtools}
\usepackage{thm-restate}
\usepackage{xparse}
\usepackage{empheq}

\usepackage{graphicx}
\graphicspath{{./pics/}}

\bibliography{new-pec-consistency}

\newcommand{\email}[1]{\href{mailto:#1}{#1}}

\allowdisplaybreaks

\newtheorem{theorem}{Theorem}
\newtheorem{proposition}[theorem]{Proposition}
\newtheorem{lemma}[theorem]{Lemma}
\newtheorem{corollary}[theorem]{Corollary}
\theoremstyle{remark}
\newtheorem{remark}[theorem]{Remark}
\theoremstyle{definition}

\newcommand{\ul}[1]{\underline{{#1}}}

\newcommand{\Real}{\mathbb{R}}

\newcommand{\st}{\,:\,}

\DeclareMathOperator{\tr}{tr}

\DeclareMathOperator{\card}{card}

\DeclareMathOperator{\Ker}{Ker}
\DeclareMathOperator{\Image}{Im}

\newcommand{\DIFF}{\mathrm{d}}
\newcommand{\KOSZUL}{\kappa}

\newcommand{\trimmed}{{-}}
\newcommand{\wh}[1]{\widehat{#1}}

\newcommand{\PL}[2]{\mathcal{P}_{#1}\Lambda^{#2}}
\newcommand{\PLtrim}[2]{\mathcal{P}_{#1}^{\trimmed}\Lambda^{#2}}

\newcommand{\lproj}[3]{\pi^{#2}_{#1,#3}}
\newcommand{\ltproj}[3]{\lproj{#1}{\trimmed,#2}{#3}}

\newcommand{\Mh}{\mathcal{M}_h}
\newcommand{\Sh}{\mathcal{S}_h}
\newcommand{\FM}[1]{\Delta_{#1}}

\newcommand{\pf}{{\partial f}}
\newcommand{\dtop}{n}

\newcommand{\ffp}{f\!f'}
\newcommand{\orffp}{\epsilon_{\ffp}}

\newcommand{\uH}[2]{\underline{X}_{r,#2}^{#1}}

\newcommand{\uI}[2]{\underline{I}_{r,#2}^{#1}}

\newcommand{\vvvert}{\vert\kern-0.25ex\vert\kern-0.25ex\vert}
\newcommand{\norm}[2]{\Vert #2\Vert_{#1}}
\newcommand{\seminorm}[2]{| #2 |_{#1}}

\newcommand{\opn}[2]{\vvvert #2\vvvert_{#1}}

\newcommand{\mc}{\mathcal}
\newcommand{\chain}[1]{C_{#1}}
\newcommand{\cochain}[1]{C^{#1}}
\newcommand{\bd}[1]{\partial_{#1}}
\newcommand{\cobd}[1]{\delta^{#1}}
\newcommand{\cycle}[1]{Z_{#1}}

\newcommand{\cspace}[1]{Z^\mathsf{c}_{#1}}

\let\boundary\relax
\newcommand\boundary[1]{B_{#1}}

\newcommand{\inner}[2]{\langle {#1} \, , {#2} \rangle}

\newcommand{\Alt}{\mathrm{Alt}}
\newcommand{\lift}[2]{{\mathcal L}^{#1}_{r,#2}}

\newcommand{\Wbf}{\phi} 
\newcommand{\PLtrimz}[2]{\mathring{\mathcal{P}}_{#1}^{\trimmed}\Lambda^{#2}} 

\newcommand{\restr}[2]{#1|_{#2}} 

\newcommand{\liftc}[2]{{\mathcal C}^{#1}_{r,#2}} 

\DeclareMathOperator{\ctr}{\gamma}

\begin{document}

\title{Conforming lifting and adjoint consistency for the Discrete de Rham complex of differential forms}
\author[1]{Daniele A. Di Pietro}
\author[1,2]{J\'{e}r\^{o}me Droniou}
\author[1]{Silvano Pitassi}
\affil[1]{IMAG, Univ Montpellier, CNRS, Montpellier, France\\
  \email{daniele.di-pietro@umontpellier.fr}, %
  \email{jerome.droniou@cnrs.fr}
  \email{silvano.pitassi@umontpellier.fr}%
}
\affil[2]{School of Mathematics, Monash University, Melbourne, Australia.
}

\maketitle

\begin{abstract}
  Discrete de Rham (DDR) methods provide non-conforming but compatible approximations of the continuous de Rham complex on general polytopal meshes.
  Owing to the non-conformity, several challenges arise in the analysis of these methods. In this work, we design conforming liftings on the DDR spaces, that are right-inverse of the interpolators and can be used to solve some of these challenges. We illustrate this by tackling the question of the global integration-by-part formula. By non-conformity of the discrete complex, this formula involves a residual -- which can be interpreted as a consistency error on the adjoint of the discrete exterior derivative -- on which we obtain, using the conforming lifting, an optimal bound in terms of the mesh size.
  Our analysis is carried out in the polytopal exterior calculus framework, which allows for unified proofs for all the spaces and operators in the DDR complex. Moreover, the liftings are explicitly constructed in finite element spaces on a simplicial submesh of the underlying polytopal mesh, which gives more control on the resulting functions (e.g., discrete trace and inverse inequalities).
  \medskip\\
  \textbf{Key words.} Discrete de Rham complex, polytopal method, conforming lifting, consistency, differential forms, exterior calculus.
  \medskip\\
  \textbf{MSC2020.} 65N30, 65N12.
\end{abstract}



\section{Introduction}

In this paper, we  address a long-standing problem in the analysis of non-conforming compatible methods, namely the construction of explicit, conforming liftings. Focusing on the Discrete de Rham complex of~\cite{Bonaldi.Di-Pietro.ea:24}, we build a piecewise polynomial lifting and showcase a first application: the rigorous proof of adjoint consistency estimates, laying another crucial brick in the construction of this complex.
Other foreseeable applications of this lifting include the proof of discrete functional analysis results for polytopal complexes such as Sobolev--Poincar\'e inequalities or discrete versions of the Rellich--Kondrachov and Maxwell compactness results.

Polytopal methods are numerical schemes for partial differential equations (PDEs) that support meshes made of generic polytopes.
They are inherently non-conforming, either because the discrete spaces they hinge on are not contained in their continuous counterpart, as in Discrete de Rham (DDR) \cite{Di-Pietro.Droniou.ea:20,Di-Pietro.Droniou:23*1}, Hybrid High-Order (HHO) \cite{Di-Pietro.Ern.ea:14,Di-Pietro.Ern:15,Di-Pietro.Droniou:20}, or Finite Volume (FV) methods \cite{Eymard.Gallouet.ea:00}, or because they require computable projections and stabilisation operators, as in Virtual Element Methods (VEM) \cite{Beirao-da-Veiga.Brezzi.ea:13,Brenner.Guan.ea:17,Beirao-da-Veiga.Mascotto.ea:22}. In either case, this lack of conformity results in the absence of exact integration by parts (IBP) formulas at the discrete level.
The residual up to which global IBP formulas are satisfied can be interpreted as an adjoint consistency error.
When the variational formulation of a continuous model is obtained from its strong formulation through IBPs, estimating adjoint consistency errors is critical to proving the consistency of the scheme.
Depending on the approach, such estimates are either explicitly stated as standalone results (see, e.g., \cite[Lemma 2.18]{Di-Pietro.Droniou:20} or \cite[Section 6.2]{Di-Pietro.Droniou:23*1}) or implicit used inside proofs of error estimates (as for standard Finite Element methods, where they are used to bound certain components of the error; see, e.g., the third term in \cite[Eq.~(2.20)]{Ern.Guermond:04} and also the error component defined by~\cite[Eq.~(27.3)]{Ern.Guermond:21}).
For methods such as HHO or FV, optimal estimates of the adjoint consistency error can be obtained by simply manipulating the formulas describing discrete operators.
For compatible methods such as DDR or VEM, on the other hand, this is a much more challenging task, which, so far, has required the construction and fine-grain study of conforming liftings obtained by solving local PDEs; see, e.g.,~\cite[Appendix~B]{Di-Pietro.Droniou:23*1}.
This has typically required subtle analysis arguments, made particularly challenging by the need to work on polytopal domains.

Our approach removes these limitations by developing computable conforming liftings on local polynomial spaces.
Such liftings are conceived with the DDR complex in mind and, for greater generality, we focus on its recent formulation based on differential forms~\cite{Bonaldi.Di-Pietro.ea:24}.
Inspired by similar developments in the context of Finite Elements~\cite{Arnold.Falk.ea:06,Arnold:18,Rapetti.Bossavit:09,Christiansen.Rapetti:16,Di-Pietro.Hanot:24}, this complex has made it possible to generalize results originally obtained for vector proxies \cite{Beirao-da-Veiga.Brezzi.ea:18,Di-Pietro.Droniou.ea:20} to domains of general dimensions~\cite{Di-Pietro.Droniou.ea:25} and also manifolds~\cite{Droniou.Hanot.ea:25}.

As already mentioned, a crucial innovative aspect of our work lies in the way we construct the conforming liftings. These liftings must satisfy two essential properties: a projection condition on cells of any dimension, ensuring compatibility with the discrete differential operators, and a boundedness condition in terms of the discrete data. We draw some inspiration from the curl-lifting technique of \cite[Section~6.5 and Appendix B]{Di-Pietro.Droniou:23*1} and by the minimal reconstruction operators for Mimetic Finite Differences of \cite[Chapter 3]{Beirao-da-Veiga.Lipnikov.ea:14}, both relying on the solution on top-dimensional cells of the continuous div-curl-type system
\[
\DIFF\lambda = \xi,\qquad
\delta\lambda = \zeta,\qquad
\tr_{\partial f}\,\lambda = \theta.
\]
where $\DIFF$ is the exterior derivative and $\delta = (-1)^k\star^{-1}\DIFF\star$ is the co-differential.
Here, we simplify this approach by omitting the co-differential constraint $\delta\lambda=\zeta$ and just select the minimal-norm solution to the resulting system, which ensures the linearity of the lifting. Crucially, when the data $\xi$ and $\theta$ are piecewise polynomial, the minimum norm solution $\lambda$ is itself piecewise polynomial, and is obtained by explicitly prescribing appropriate degrees of freedom.
This stands in stark contrast with previous approaches, and removes the limitations therein coming from the use of continuous PDEs (most notably, the fact that some fine regularity properties only hold in dimension $\le 3$).

The original DDR complex of~\cite{Di-Pietro.Droniou:23*1} has also be interpreted as a VEM complex \cite{Beirao-da-Veiga.Dassi.ea:22}, albeit different from the ones encountered in classical VEM literature. This presentation uncovers the corresponding virtual spaces, which are spanned by conforming functions obtained from the DDR degrees of freedom (DOFs).
Virtual functions, however, cannot be used as liftings in the proof of adjoint consistency, as no clear result bounding their norms in terms of the corresponding DOFs is available.
Such bounds on virtual functions are indeed not easy to establish, even in the simpler scalar case corresponding to $0$-forms \cite{Brenner.Guan.ea:17,Beirao-da-Veiga.Lovadina.ea:17}. The situation is even more dire for the vector case corresponding to $1$-forms, where fine estimates for div-curl PDEs on polytopal domains are required; see~\cite{Beirao-da-Veiga.Mascotto.ea:22}.
On the other hand, the VEM complex of \cite{Beirao-da-Veiga.Brezzi.ea:18} has been, in turn, recast in a unified language of differential forms in \cite{Bonaldi.Di-Pietro.ea:24}, and we believe that our approach to construct conforming liftings could also be adapted to this complex, and possibly others.

We conclude by pointing out two other possible future applications of the conforming lifting developed in the present paper: general Poincaré--Sobolev inequalities for all form degrees, and compactness results for compatible polytopal methods.
Such results will unlock rigorous analyses for nonlinear problems, including proofs of convergence to minimal regularity solutions that have been completely missing in the literature so far.

The rest of the paper is organized as follows.
In Section \ref{sec:presentation} we introduce notations and recall construction of the DDR complex in its exterior calculus presentation. The principal results are presented in Section \ref{sec:main.results}: existence and main properties of the conforming lifting (Theorem \ref{thm:conforming.lifting}), and adjoint consistency for the DDR -- in wedge and inner product formulations (Theorem \ref{thm:adjoint.consistency} and Corollary \ref{cor:adjoint.consistency.inner}). The proof of the consistency properties are gathered in Section \ref{sec:proof.consistency}. The majority of the article, Section \ref{sec:construction.lifting}, is devoted to the construction of the lifting, starting with preliminaries and local problem in cochain spaces (Section \ref{sec:cochains}), before analysing two local problems in trimmed polynomial spaces (Section \ref{sec:local.problems}), which are then used in a hierarchical construction of the lifting on the cells dimension (Section \ref{sec:def.lifting}), while the last section (Section \ref{sec:properties.lifting}) checks that the constructed lifting satisfies the required properties.


\section{The Discrete de Rham complex}\label{sec:presentation}

\subsection{Setting}\label{sec:presentation:mesh}

We give here a brief presentation of the setting for the polytopal exterior calculus Discrete de Rham (DDR) complex, and refer the reader to \cite{Bonaldi.Di-Pietro.ea:24} for more details.

\subsubsection{Polytopal mesh}\label{sec:mesh}

We consider a bounded polytopal domain $\Omega\subset\Real^\dtop$, discretised by a polytopal mesh $\Mh$. This mesh is a finite collection of (closed) polytopes of dimensions $d\in\{0,1,\ldots,\dtop\}$ called ($d$-)cells. The boundary of any $d$-cell is a union of $(d-1)$-cells and the intersection of two $d$-cells is either empty or a union of cells of lower dimension. The set of $d$-cells is denoted by $\FM{d}(\Mh)$, and we let $\FM{d}(f)$ be the set of $d$-cells contained in a given cell $f$ of dimension $\ge d$. We similarly write $\FM{d}(\partial f)$ for the set of $d$-cells contained in $\partial f$.
Each $f\in\Mh$ of dimension $\ge 1$ is equipped with an orientation, and we denote by $h_f$ its diameter.

A $d$-simplex is the image by an affine map of the reference simplex $\{(x_0,\ldots,x_d)\in\Real^d\,:\,x_i\ge 0\,,\;\sum_{i=0}^dx_i=1\}$. The concept of (conforming) simplicial mesh is understood in the usual sense. We will assume that $\Mh$ is regular in a sense that generalises \cite[Definition~1.9]{Di-Pietro.Droniou:20}. In particular, this means that $\Mh$ has a simplicial sub-mesh $\Sh$ -- that is, each $f\in\FM{d}(\Mh)$ is the union of a set $\Sh(f)$ of $d$-simplices in $\Sh$ -- and that there exists $\rho\in (0,1)$ independent of $h$ such that: (i) for any simplex $F\in\Sh$ with inradius $r_F$ and diameter $h_F$, we have $\rho h_F\le r_F$, and (ii) for any $f\in\Mh$ of dimension $\ge 1$ and any $F\in\Sh(f)$, $\rho h_f\le h_F$.
For simplicity, if $f$ is a $d$-cell we denote by $\Sh(\partial f)=\cup_{f'\in\FM{d-1}(f)}\Sh(f')$ the set of simplices of $\Sh$ contained in $\partial f$.
Notice that $\Sh(f)$ and $\Sh(\pf)$ form simplicial meshes on $f$ and on $\pf$, respectively.

From here on, we will write $a\lesssim b$ as a shorthand for $a\le Cb$ with $C$ depending on $\Omega$ and $\rho$, but not depending on $h$ nor on a specific cell in $\Mh$ or simplex in $\Sh$.
The mesh regularity assumptions above imply $\card(\Sh(f))\lesssim 1$ for all $f\in\Mh$.

\subsubsection{Differential forms}

If $f$ is a $d$-cell and $k \in \{0, \dots, d\}$ is a form degree, $\Lambda^k(f)$ denotes the set of differential $k$-forms on $f$, that
is, mappings $\omega:f\to \Alt^k(f)$ where $\Alt^k(f)$ is the set of alternating $k$-multilinear forms on the tangent space of $f$
(vector space of the affine span of $f$).
Any specific regularity on these forms (i.e., on the coefficients of their expansion in a given basis of $\Alt^k(f)$) is denoted by prepending $\Lambda^k(f)$; for example, $C^0\Lambda^k(f)$ and $H^1\Lambda^k(f)$ respectively denote the spaces of forms with coefficients that are continuous or in $H^1(f)$.

We additionally introduce the space $H\Lambda^k(f)$ of square-integrable forms whose exterior derivative is also square-integrable:
\[
H\Lambda^k(f)\coloneq\{\alpha\in L^2\Lambda^k(f)\,:\,\DIFF\alpha\in L^2\Lambda^k(f)\}.
\]
Note that there is no ambiguity with, e.g., $H^1\Lambda^k(f)$ since $H$ above does not carry any exponent.

The Hodge-star operator on a $d$-cell $f$ is denoted by $\star_f$; it is an isomorphism $\Lambda^k(f)\to \Lambda^{d-k}(f)$ for all
$d \in\{0,\ldots,k\}$. As the domain can be inferred from the context, we will drop the index $f$ and only write $\star$.

For each $f\in\Mh$, we fix a point ${x}_f\in f$ such that $f$ contains a ball centered at $x_f$ and of radius $\gtrsim h_f$. 
The Koszul operator $\KOSZUL_f$ on $f$ is then defined as the interior product against the identity vector field shifted at $x_f$:
if $\omega \in \Lambda^k(f)$, then $\KOSZUL_f\omega \in \Lambda^{k-1}(f)$ is defined by 
\[
(\KOSZUL_f\omega)_{{x}}({v}_1, \dots, {v}_{k-1}) 
= \omega_{{x}}({x}-{x}_f,{v}_1, \dots, {v}_{k-1}),
\]
for all ${x}\in f$ and all ${v}_1, \dots, {v}_{k-1}$ tangent vectors to $f$.

The inner product on $L^2\Lambda^k(f)$ is given by
\begin{equation}\label{eq:def.L2.forms}
\langle \omega,\mu\rangle_f\coloneq \int_f \omega\wedge\star\mu
\end{equation}
and the corresponding norm is denoted by $\norm{f}{\omega} \coloneq \langle \omega, \omega\rangle_f^{\frac12}$.

\subsubsection{Local polynomial forms}

If $r\ge 0$ is an integer, $\PL{r}{k}(f)$ is the space of $k$-forms on $f$ whose coefficient on a basis of $\Alt^k(f)$ are polynomials
of (total) degree $\le r$. By convention, we set $\PL{r}{k}(f) \coloneq \left\{ 0 \right\}$ when $r<0$.
The following decompositions of polynomial space hold (see, e.g., \cite[Eq.~(3.11)]{Arnold.Falk.ea:06} for $k\ge 1$):
\begin{equation}\label{eq:decomp.PL}
  \begin{aligned}
    \PL{r}{0}(f) &= \PL{0}{0}(f) \oplus \KOSZUL_f \PL{r-1}{1}(f),\\
    \PL{r}{k}(f) &= \DIFF \PL{r+1}{k-1}(f) \oplus \KOSZUL_f \PL{r-1}{k+1}(f)\qquad\text{if $k\ge 1$}.
  \end{aligned}
\end{equation}
The spaces of \emph{trimmed} polynomials are obtained by lowering the degree of the first term in this decomposition for $k\ge 1$:
\begin{equation}    \label{eq:def.Ptrim}
  \begin{aligned}
    \PLtrim{r}{0}(f) &\coloneq \PL{r}{0}(f), \\
    \PLtrim{r}{k}(f) &\coloneq \DIFF \PL{r}{k-1}(f) \oplus \KOSZUL_f \PL{r-1}{k+1}(f) \qquad\text{ if } k\ge 1.
  \end{aligned}
\end{equation}
The $L^2$-orthogonal projector $\ltproj{r}{k}{f} \st L^2\Lambda^k(f) \to \PLtrim{r}{k}(f)$ is such that, for $\omega \in L^2\Lambda^k(f)$,
\[
\langle \ltproj{r}{k}{f} \omega, \mu \rangle_f = \langle \omega, \mu \rangle_f \qquad \forall \mu \in \PLtrim{r}{k}(f).
\]

\subsection{The Discrete de Rham construction}\label{sec:ddr}

Let us fix an integer $r \ge 0$ corresponding to the polynomial degree of the discrete de Rham complex.

\subsubsection{Spaces, interpolators and discrete operators}

For any $k \in \{0, \ldots, n\}$, the discrete version $\uH{k}{h}$ of $H\Lambda^k(\Omega)$ is
\[
\uH{k}{h} \coloneq \bigtimes_{d = k}^{n} \bigtimes_{f \in \FM{d}(\Mh)} \star^{-1}\PLtrim{r}{d-k}(f).
\]

\begin{remark}[Choice of component spaces]\label{rem:choice.component.space}
  Notice that, as in \cite{Di-Pietro.Droniou.ea:25}, we have taken $\star^{-1}\PLtrim{r}{d-k}(f)$ instead of $\PLtrim{r}{d-k}(f)$ as component space. This choice allows for an easier generalisation to manifolds \cite{Droniou.Hanot.ea:25}.
  In the context of flat spaces as considered here the results of \cite{Bonaldi.Di-Pietro.ea:24} can still be invoked with minimal adaptations. Moreover, by \eqref{eq:def.Ptrim}, we have
  \begin{equation}\label{eq:star.P.P}
    \star^{-1}\PLtrim{r}{d-k}(f) =\left\{\begin{array}{ll}
    \PL{r}{k}(f)&\text{ if $d=k$},\\
    \PLtrim{r}{k}(f)&\text{ if $d>k$}.
    \end{array}\right.
  \end{equation}
\end{remark}

The component of $\ul{\omega}_h\in \uH{k}{h}$ on a $d$-cell $f\in\Mh$ with $d\ge k$ is denoted by $\omega_f$. The vector $\ul{\omega}_f=(\omega_{f'})_{f'\in\FM{d'}(f),\,d'\in[k,d]}$ is the restriction of $\ul{\omega}_h$ to $f$, which gathers all the components on $f$ and the cells in its boundary. $\uH{k}{f}$ is the space of such restrictions.
The local interpolator $\uI{k}{f} : C^0\Lambda^k(f) \to \uH{k}{f}$ is defined through the $L^2$-projections on trimmed polynomial spaces by
\begin{equation}\label{eq:local.interpolator}
  \uI{k}{f} \omega \coloneq \left(
  \star^{-1}\ltproj{r}{d'-k}{f'}\star \tr_{f'} \omega
  \right)_{f' \in \FM{d'}(f),\, d' \in \{k,\ldots,d\}}
  \qquad \forall \omega \in C^0\Lambda^k(f).
\end{equation}
The global interpolator $\uI{k}{h}:C^0\Lambda^k(\overline{\Omega})\to\uH{k}{h}$ is obtained by gathering all the components of the
local interpolators, for all $f\in\Mh$.

For all $d \ge k$ and all $f \in \FM{d}(\Mh)$, we now define a discrete potential $P^k_{r,f} \st \uH{k}{f} \rightarrow \PL{r}{k}(f)$
and, if $d \ge k+1$, a discrete exterior derivative $\DIFF^k_{r,f} \st \uH{k}{f} \rightarrow \PL{r}{k+1}(f)$.
In what follows, for sake of brevity we set
\begin{equation*}
  \int_{\pf} \bullet \coloneq \sum_{f'\in\FM{d-1}(\partial f)}\orffp\int_{f'}\bullet,
\end{equation*}
where $\orffp\in\left\{ -1,+1 \right\}$ is the orientation of $f'$ relative to $f$. The discrete potential and discrete exterior derivative are then defined recursively as follows: for $\ul{\omega}_h\in\uH{k}{h}$,
\begin{itemize}
\item If $d = k$, 
  \begin{equation} \label{eq:def.P.d=k}
    P^k_{r,f} \ul{\omega}_f
    \coloneq \omega_f \in \star^{-1}\PL{r}{0}(f)
    \overset{\eqref{eq:star.P.P}}=\PL{r}{k}(f).
  \end{equation}
\item For $d \in \{ k+1,\ldots,n\}$:
  \begin{enumerate}
  \item We first define the discrete exterior derivative such that
    \begin{multline} \label{eq:def.d}
      \int_f \DIFF^k_{r,f} \ul{\omega}_f \wedge \mu
      = (-1)^{k+1} \int_f \omega_f \wedge \DIFF \mu
      + \int_{\pf} P^k_{r,\pf} \ul{\omega}_f \wedge \tr_\pf \mu
      \\
      \forall \mu \in \PL{r}{d-k-1}(f),
    \end{multline}
    where $P^k_{r,\pf}\ul{\omega}_f$ is the piecewise polynomial form on $\pf$ obtained patching together the already constructed polynomials $(P^k_{r,f'}\ul{\omega}_{f'})_{f'\in\FM{d-1}(\partial f)}$. 
  \item We then define the discrete potential:
    \begin{multline} \label{eq:def.pot}
      (-1)^{k+1} \int_f P^k_{r,f} \ul{\omega}_f\wedge (\DIFF \mu + \nu )
      \\
      = \int_f \DIFF^k_{r,f} \ul{\omega}_f \wedge  \mu 
      - \int_{\pf} P^k_{r,\pf} \ul{\omega}_\pf \wedge \tr_\pf \mu
      + (-1)^{k+1} \int_f  \omega_f \wedge \nu 
      \\
      \forall (\mu, \nu) \in \KOSZUL_f \PL{r}{d-k}(f) \times \KOSZUL_f \PL{r-1}{d-k+1}(f).
    \end{multline}
  \end{enumerate}
\end{itemize}

\begin{remark}[Validity of \eqref{eq:def.pot}]
  The definition \eqref{eq:def.d} of $\DIFF^k_{r,f}$ and the relation $\DIFF^2=0$ show that both sides of \eqref{eq:def.pot} vanishes when $\nu=0$ and $\mu\in\DIFF \PL{r+1}{d-k-2}(f)+\PL{0}{d-k-1}(f)$. Using \eqref{eq:def.Ptrim} (together with \eqref{eq:decomp.PL} if $d-k-1=0$) with $r\leftarrow r+1$ and combining with \eqref{eq:def.pot}, we infer
  \begin{equation}\label{eq:ipp.pot}
    (-1)^{k+1}\int_f P^k_{r,f}\ul{\omega}_f\wedge \DIFF \mu
    = \int_f \DIFF^k_{r,f}\ul{\omega}_f \wedge \mu
    - \int_{\pf} P^k_{r,\pf} \ul{\omega}_{\pf} \wedge \tr_{\pf} \mu\qquad
    \forall \mu\in \PLtrim{r+1}{d-k-1}(f).
  \end{equation}
\end{remark}

The global discrete exterior derivative $\ul{\DIFF}^k_{r,h}:\uH{k}{h}\to \uH{k+1}{h}$ is obtained by projecting each discrete exterior derivative on the corresponding component of $\uH{k+1}{h}$:
\begin{equation*} 
  \ul{\DIFF}^k_{r,h} \ul{\omega}_h \coloneq \left(
  \star^{-1}\ltproj{r}{d-k-1}{f}\star \DIFF^k_{r,f} \ul{\omega}_f
  \right)_{f \in \FM{d}(\Mh), d \in \{k+1, \dots, \dtop\}}.
\end{equation*}
The DDR complex then reads
\begin{equation*}
  \begin{tikzcd}
    0 \arrow{r}{}
    &\uH{0}{h} \arrow{r}{\ul{\DIFF}^0_{r,h}}
    &\uH{1}{h} \arrow{r}{\ul{\DIFF}^1_{r,h}}
    &\cdots \arrow{r}{}
    &\uH{n-1}{h} \arrow{r}{\ul{\DIFF}^{n-1}_{r,h}}
    &\uH{n}{h} \arrow{r}{}
    &0.
  \end{tikzcd}
\end{equation*}

\subsubsection{Norms and discrete inner products}

For any form degree $k \in \{0, \dots, \dtop\}$ and any $\ul{\omega}_h \in \uH{k}{h}$, we define a component $L^2$-like norm recursively:
\begin{subequations}\label{eq:opn.f}
  \begin{itemize}
  \item For each $f \in \FM{k}(\Mh)$, recalling that $\norm{f}{{\cdot}}$ is the norm on $L^2\Lambda^k(f)$, we set
    \begin{equation}\label{eq:opn.f.k}
      \opn{f}{\ul{\omega}_f} \coloneq \norm{f}{\omega_f}.
    \end{equation}
  \item For $d \in \{ k+1,\dots,\dtop\}$ and $f\in\FM{d}(\Mh)$,
    \begin{equation}\label{eq:opn.f.k+}
      \opn{f}{\ul{\omega}_f} \coloneq \left(
      \norm{f}{\omega_f}^2 +
      h_f \sum_{f'\in\FM{d-1}(f)} \opn{f'}{\omega_{f'}}^2
      \right)^{\frac12}.
    \end{equation}
  \end{itemize}
\end{subequations}
The global norm is then defined as
\begin{equation*} 
  \opn{h}{\ul{\omega}_h} \coloneq \left(
  \sum_{f\in\FM{\dtop}(\Mh)}\opn{f}{\ul{\omega}_h}^2
  \right)^{\frac12}.
\end{equation*}

The discrete potential enables us to also define an $L^2$-like inner product on local spaces $\uH{k}{f}$, for $f\in\FM{d}(\Mh)$ with $d\ge k$:
\begin{equation}\label{eq:def.local.inner.product}
  (\ul{\omega}_f,\ul{\mu}_f)_{k,f}\coloneq
  \langle P^k_{r,f}\ul{\omega}_f, P^k_{r,f}\ul{\mu}_f\rangle_f
  + \mathrm{s}_{k,f}(\ul{\omega}_f,\ul{\mu}_f)\qquad
  \forall \ul{\omega}_f,\ul{\mu}_f\in\uH{k}{f},
\end{equation}
where $\mathrm{s}_{k,f}$ is a positive semi-definite bilinear form on $\uH{f}{k}$ that penalises the difference between the discrete potential on $f$ and those on $f'\in\cup_{d'\in\{k,\ldots,d-1\}}\FM{d'}(f)$, and vanishes on interpolants of polynomial forms of degree $\le r$ on $f$; see \cite[Section 3.1.5]{Bonaldi.Di-Pietro.ea:24} for further details. The stabilisation is in particular selected to ensure that the norm associated with this inner product is equivalent to the component norm:
\begin{equation}\label{eq:norm.equivalence}
  \opn{f}{{\cdot}}\simeq (\cdot,\cdot)_{k,f}^{1/2}.
\end{equation}

\section{Main results}\label{sec:main.results}

\subsection{Conforming lifting}

\begin{theorem}[Conforming lifting]\label{thm:conforming.lifting}
  Let $\Mh$ be a polytopal mesh, $r\ge 0$ be a polynomial degree, and $k\in\{0,\ldots,\dtop\}$ be a form degree.
  Then, there exists a linear lifting $\lift{k}{h}:\uH{k}{h}\to H\Lambda^{k}(\Omega)$ such that, for all $\ul{\omega}_h\in\uH{k}{h}$,
  all $d \in \{k, \dots, \dtop\}$ and all $f\in\FM{d}(\Mh)$,
  \begin{align}
    \label{eq:lift.proj}
    \int_f \tr_f(\lift{k}{h} \ul{\omega}_h) \wedge \zeta
    ={}&\int_f P^{k}_{r,f} \ul{\omega}_f \wedge \zeta \qquad \forall \zeta \in \PLtrim{r+1}{d-k}(f),\\
    \label{eq:dlift.proj}
    \int_f \DIFF\tr_f(\lift{k}{h} \ul{\omega}_h) \wedge \mu
    ={}&
      \int_f \DIFF^{k}_{r,f} \ul{\omega}_f \wedge \mu \qquad \forall \mu \in \PLtrim{r+1}{d-k-1}(f)\quad\text{(for $d\ge k+1$)},
    \\
    \label{eq:lift.bound.0}
    \norm{f}{\tr_f(\lift{k}{h} \ul{\omega}_h)} \lesssim{}& \opn{f}{\ul{\omega}_f},\\
    \label{eq:lift.bound.1}
    \norm{f}{\DIFF \tr_f(\lift{k}{h} \ul{\omega}_h)}
    \lesssim{}& \opn{f}{\ul{\DIFF}^{k}_{r,f}\ul{\omega}_f}\qquad\text{ if $d\ge k+1$}.
  \end{align}
\end{theorem}

\begin{proof} The lifting is defined in Section \ref{sec:def.lifting} and its properties are established in Section \ref{sec:properties.lifting}.
\end{proof}

\begin{remark}[Piecewise polynomial lifting]\label{rem:piecewise.lifting}
  The construction of the lifting shows that, for all $\ul{\omega}_h\in\uH{k}{h}$ and all $f\in\FM{d}(\Mh)$ with $d\ge k$,
  the trace $\tr_f(\lift{k}{h}\ul{\omega}_h)$ is well-defined and single-valued, and belongs to $\PLtrim{r+1+d-k}{k}(F)$ for all $F\in\FM{d}(\Sh(f))$.
  In particular, the lifting is piecewise polynomial on $\Sh$.
\end{remark}

\begin{remark}[The lifting is a right-inverse of the interpolator]
  The definition \eqref{eq:def.L2.forms} of the $L^2$-inner product on differential forms, the relation $\rho\wedge\nu=\star\rho\wedge\star\nu$, and the fact that $\star P^k_{r,f}\ul{\omega}_f\in\PL{r}{d-k}(f)\subset \PLtrim{r+1}{d-k}(f)$ and $\star \DIFF^k_{r,f}\ul{\omega}_f\in\PL{r}{d-k-1}(f)\subset \PLtrim{r+1}{d-k-1}(f)$ show that \eqref{eq:lift.proj} and \eqref{eq:dlift.proj} are equivalent to
  \begin{subequations}\label{eq:proj.L.P.d}
  \begin{align}\label{eq:projection.lifting.pot}
    \ltproj{r+1}{d-k}{f}(\star\tr_f\lift{k}{h}\ul{\omega}_h)
    &= \star P^k_{r,f}\ul{\omega}_f,\\
    \label{eq:projection.diff.lifting.pot}
      \ltproj{r+1}{d-k-1}{f}(\star\DIFF\tr_f\lift{k}{h}\ul{\omega}_h)
    &=
      \star \DIFF^k_{r,f}\ul{\omega}_f\quad\text{(for $d\ge k+1$)}.
  \end{align}
  \end{subequations}
  By \cite[Eq.~(3.37)]{Bonaldi.Di-Pietro.ea:24} and the choice of space commented in Remark \ref{rem:choice.component.space}, we have $\ltproj{r}{d-k}{f}(\star P^k_{r,f}\ul{\omega}_f)=\star\omega_f$. Applying $\ltproj{r}{d-k}{f}$ to both sides of \eqref{eq:projection.lifting.pot}, noticing that $\ltproj{r}{d-k}{f}\circ \ltproj{r+1}{d-k}{f}=\ltproj{r}{d-k}{f}$, and recalling the definition \eqref{eq:local.interpolator} of the local interpolator, we see that for all $\ul{\omega}_h\in\uH{k}{h}$, the interpolator can be applied to $\lift{k}{h}\ul{\omega}_h$ (which is not in $C^0 \Lambda^k(\overline{\Omega})$ but has single-valued integrable traces on lower-dimensional cells, see Remark \ref{rem:piecewise.lifting}), and that
  \[
  \uI{k}{h}\lift{k}{h}\ul{\omega}_h=\ul{\omega}_h.
  \]
\end{remark}

\begin{remark}[Homogeneous boundary conditions]\label{rem:homogeneous.BC}
  The estimate \eqref{eq:lift.bound.0} shows that the lifting preserves local zero traces, in the sense that
  if $\ul{\omega}_f=0$ for some $f\in\FM{\dtop-1}(\Mh)$, then $\tr_f(\lift{k}{h}\ul{\omega}_h)=0$. In particular,
  if $\ul{\omega}_h$ vanishes on $\partial\Omega$ (in the sense that all its components on cells included in $\partial\Omega$
  vanish), then $\tr_{\partial\Omega}\lift{k}{h}\ul{\omega}_h=0$; in other words, the lifting preserves homogeneous boundary conditions.
\end{remark}

\subsection{Consistency properties}\label{sec:adjoint.consistency}

We first recall some primal consistency properties. For this, we introduce the notation, for $\omega\in H^{\max(r+1,s)}\Lambda^k(f)$,
\[
\seminorm{H^{(r+1,s)}\Lambda^k(f)}{\omega}\coloneq
\left\{\begin{array}{cc}
\displaystyle\seminorm{H^{r+1}\Lambda^k(f)}{\omega}&\mbox{ if $s\le r+1$},\\[.3em]
\displaystyle\sum_{t=r+1}^s h_f^{t-r-1}\seminorm{H^{t}\Lambda^k(f)}{\omega}&\mbox{ if $s> r+1$}.
\end{array}\right.
\]

\begin{proposition}[Primal consistency]
  Let $k\in\{0,\ldots,\dtop\}$, $d\ge k$ and $f\in\Delta_d(\Mh)$. Then, for any integers $s\ge 0$ such that $2s>d$ and $0\le m\le r+1$, it holds:
  \begin{enumerate}
  \item Consistency of the discrete potential:
    \begin{equation}\label{eq:discrete.potential:polynomial.consistency.smooth}
      \seminorm{H^m\Lambda^k(f)}{P^k_{r,f}\uI{k}{f}\omega - \omega}\lesssim h_f^{r+1-m}\seminorm{H^{(r+1,s)}\Lambda^k(f)}{\omega}
      \qquad\forall\omega\in H^{\max(r+1,s)}\Lambda^k(f).
    \end{equation}
  \item Consistency of the discrete exterior derivative: if $d\ge k+1$,
    \begin{multline}\label{eq:discrete.exterior.derivative:polycons.smooth}
      \seminorm{H^{m}\Lambda^{k+1}(f)}{\DIFF_{r,f}^k\uI{k}{f}\omega - \DIFF\omega}\lesssim h_f^{r+1-m}\seminorm{H^{(r+1,s)}\Lambda^{k+1}(f)}{\DIFF\omega}
      \\
      \forall\omega\in C^1\Lambda^k(f)\mbox{ s.t. }\DIFF\omega\in H^{\max(r+1,s)}\Lambda^{k+1}(f).
    \end{multline}
  \item Consistency of the inner product:
    \begin{multline}\label{eq:inner.product.consistency}
      \left|(\uI{k}{h}\omega,\ul{\mu}_f)_{k,f}-\int_f \omega\wedge\star P^{k}_{r,f}\ul{\mu}_f\right|\lesssim
      h_f^{r+1}\seminorm{H^{(r+1,s)}\Lambda^k(f)}{\omega}\opn{f}{\ul{\mu}_f}\\
      \forall \omega\in H^{\max(r+1,s)}\Lambda^k(f)\,,\quad\forall \ul{\mu}_f\in\uH{k}{f}.
    \end{multline}
  \end{enumerate}
\end{proposition}

\begin{proof}
  The estimates \eqref{eq:discrete.potential:polynomial.consistency.smooth} and \eqref{eq:discrete.exterior.derivative:polycons.smooth}
  are established in \cite[Corollary 17]{Bonaldi.Di-Pietro.ea:24}. The estimate \eqref{eq:inner.product.consistency} is an immediate
  consequence of \eqref{eq:discrete.potential:polynomial.consistency.smooth} with $m=0$, of the consistency of the stabilisation bilinear
  form $\mathrm{s}_{k,f}$ stated in \cite[Lemma 19]{Bonaldi.Di-Pietro.ea:24}, and of the norm equivalence \eqref{eq:norm.equivalence}.
\end{proof}

To state the adjoint consistency property, we introduce the global piecewise potential $P^k_{r,h}:\uH{k}{h}\to L^2\Lambda^k(\Omega)$ obtained by patching the discrete potentials on top-dimensional cells: for all $\ul{\omega}_h \in \uH{k}{h}$, $(P^k_{r,h}\ul{\omega}_h)|_f \coloneqq P^k_{r,f}\ul{\omega}_f$ for all $f\in\FM{\dtop}(\Mh)$. Likewise, the global piecewise exterior derivative $\DIFF^k_{r,h}:\uH{k}{h}\to L^2\Lambda^{k+1}(\Omega)$ is obtained patching the discrete exterior derivatives $\DIFF^k_{r,f}$, for $f\in\FM{\dtop}(\Mh)$. Finally, $H^s\Lambda^k(\Mh)$ denotes the space of piecewise $H^s$-forms on the top-dimensional $\dtop$-cells, equipped with the norm
\begin{equation}\label{eq:def.Hr.global}
  \seminorm{H^s\Lambda^k(\Mh)}{\alpha}\coloneq\left(\sum_{f\in\FM{\dtop}(\Mh)}\seminorm{H^s\Lambda^k(f)}{\alpha}^2\right)^{1/2}.
\end{equation}

\begin{theorem}[Adjoint consistency]\label{thm:adjoint.consistency}
  Let $k\in\{0,\ldots,\dtop-1\}$, and $\alpha\in C^0\Lambda^k(\overline{\Omega})\cap H^{r+2}\Lambda^k(\Mh)$ be such that $\tr_{\partial\Omega}\alpha=0$. Then, for all $\ul{\omega}_h\in\uH{n-k-1}{h}$,
  \begin{multline}\label{eq:adjoint.consistency}
    \left|\int_\Omega \alpha\wedge \DIFF^{n-k-1}_{r,h}\ul{\omega}_h +(-1)^k \int_\Omega \DIFF\alpha\wedge P^{n-k-1}_{r,h}\ul{\omega}_h\right|
    \\
    \lesssim h^{r+1}\left(\seminorm{H^{r+1}\Lambda^k(\Mh)}{\alpha}\opn{h}{\ul{\DIFF}^{n-k-1}_{r,h}\ul{\omega}_h}
    +\seminorm{H^{r+1}\Lambda^{k}(\Mh)}{\DIFF\alpha}\opn{h}{\ul{\omega}_h}\right).
  \end{multline}
\end{theorem}

\begin{proof}
  See Section \ref{sec:proof.consistency}.
\end{proof}

In practical situations, a version of adjoint consistency written in terms of discrete inner products can also be useful. In the following corollary, $(\cdot,\cdot)_{k,h}$ denotes the global $L^2$-inner product on $\uH{k}{h}$ obtained by assembling the local inner products \eqref{eq:def.local.inner.product} on top-dimensional cells:
\[
(\ul{\omega}_h,\ul{\mu}_h)_{k,h}\coloneq\sum_{f\in\FM{\dtop}(\Mh)}(\ul{\omega}_f,\ul{\mu}_f)_{k,f}.
\]

\begin{corollary}[Adjoint consistency with discrete inner product]\label{cor:adjoint.consistency.inner}
  Let $k\in\{1,\ldots,n\}$ and $s\ge 1$ be an integer such that $2s>\dtop$. Assume that $\zeta\in H^s\Lambda^k(\Omega)\cap H^{r+2}\Lambda^k(\Mh)$ is such that $\tr_{\partial\Omega}\zeta=0$. Then, for all $\ul{\mu}_h\in\uH{k-1}{h}$,
  \begin{multline}\label{eq:adjoint.consistency.inner}
    \left|(\uI{k}{h}\zeta, \ul{\DIFF}^{k-1}_{r,h}\ul{\mu}_h)_{k,h}-\int_\Omega \delta\zeta\wedge \star P^{k-1}_{r,h}\ul{\mu}_h\right|
    \\
    \lesssim h^{r+1}\left[\left(\seminorm{H^{(r+1,s)}\Lambda^k(\Mh)}{\zeta}+\seminorm{H^{r+1}\Lambda^k(\Mh)}{\zeta}\right)\opn{h}{\ul{\DIFF}^{k-1}_{r,h}\ul{\mu}_h}+\seminorm{H^{r+1}\Lambda^{k+1}(\Mh)}{\delta\zeta}\opn{h}{\ul{\mu}_h}\right],
  \end{multline}
  where $\delta=(-1)^k\star^{-1}\DIFF\star$ is the co-differential, and the global norm $\seminorm{H^{(r+1,s)}\Lambda^k(\Mh)}{{\cdot}}$ is defined in a similar way as \eqref{eq:def.Hr.global}.
\end{corollary}
\begin{proof}
  See Section \ref{sec:proof.consistency}.
\end{proof}

\section{Proof of the consistency properties}\label{sec:proof.consistency}

We are now ready to prove the adjoint consistency estimates.

\begin{proof}[Proof of Theorem \ref{thm:adjoint.consistency}]
The assumption $\alpha\in C^0\Lambda^k(\overline{\Omega})\cap H^{r+2}\Lambda^k(\Mh)$ ensures that $\alpha\in H\Lambda^k(\Omega)$.
Since $\lift{n-k-1}{h}\ul{\omega}_h\in H\Lambda^{n-k-1}(\Omega)$ and $\tr_{\partial\Omega}\alpha=0$, the Stokes formula gives
\[
\int_\Omega \alpha\wedge \DIFF\lift{n-k-1}{h}\ul{\omega}_h+(-1)^{k}\int_\Omega \DIFF\alpha\wedge \lift{n-k-1}{h}\ul{\omega}_h=0.
\]
Subtracting this quantity from the term in the absolute value in the right-hand side of \eqref{eq:adjoint.consistency}, we infer
\begin{align*}
\int_\Omega \alpha\wedge {}&\DIFF^{n-k-1}_{r,h}\ul{\omega}_h+(-1)^{k}\int_\Omega \DIFF\alpha\wedge P^{n-k-1}_{r,h}\ul{\omega}_h\\
={}&\int_\Omega \alpha\wedge (\DIFF^{n-k-1}_{r,h}\ul{\omega}_h-\DIFF\lift{n-k-1}{h}\ul{\omega}_h)
+(-1)^{k}\int_\Omega \DIFF\alpha\wedge (P^{n-k-1}_{r,h}\ul{\omega}_h-\lift{n-k-1}{h}\ul{\omega}_h)\\
={}&
\int_\Omega \alpha\wedge \star^{-1}\left[\ltproj{r+1}{k}{h}(\star\DIFF\lift{n-k-1}{h}\ul{\omega}_h)-\star\DIFF\lift{n-k-1}{h}\ul{\omega}_h\right]\\
& +(-1)^{k}\int_\Omega \DIFF\alpha\wedge \star^{-1}\left[\ltproj{r+1}{k+1}{h}(\star\lift{n-k-1}{h}\ul{\omega}_h)-\star\lift{n-k-1}{h}\ul{\omega}_h\right]
\end{align*}
where the conclusion is obtained applying \eqref{eq:proj.L.P.d} with $(d,k)\leftarrow (n,n-k-1)$, defining $\ltproj{r+1}{\ell}{h}:L^2\Lambda^\ell(\Omega)\to \PLtrim{r+1}{\ell}(\Mh)$ as the global $L^2$-orthogonal projector on the space of broken trimmed polynomial $\ell$-forms (this projectors consists in patching together the local projectors). Using the relation $\star^{-1}\zeta=(-1)^{\ell(n-\ell)}\star\zeta$ for an $\ell$-form $\zeta$, the definition \eqref{eq:def.L2.forms} of the $L^2$-inner product and the $L^2$-self-adjoint property of the $L^2$-orthogonal projector, we infer
\begin{align*}
\int_\Omega \alpha\wedge {}&\DIFF^{n-k-1}_{r,h}\ul{\omega}_h+(-1)^{k}\int_\Omega \DIFF\alpha\wedge P^{n-k-1}_{r,h}\ul{\omega}_h\\
={}&
\int_\Omega \left[\ltproj{r+1}{k}{h}\alpha-\alpha\right]\wedge \DIFF\lift{n-k-1}{h}\ul{\omega}_h
 +(-1)^{k}\int_\Omega \left[\ltproj{r+1}{k+1}{h}(\DIFF\alpha)-\DIFF\alpha\right]\wedge \lift{n-k-1}{h}\ul{\omega}_h.
\end{align*}
Invoking the Cauchy--Schwarz inequality as well as the bounds \eqref{eq:lift.bound.0}--\eqref{eq:lift.bound.1} on the lifting then leads to
\begin{align*}
\Bigg|\int_\Omega \alpha\wedge {}&\DIFF^{n-k-1}_{r,h}\ul{\omega}_h+(-1)^{k}\int_\Omega \DIFF\alpha\wedge P^{n-k-1}_{r,h}\ul{\omega}_h\Bigg|\\
\lesssim{}&
\norm{L^2\Lambda^k(\Omega)}{\ltproj{r+1}{k}{h}\alpha-\alpha} \opn{h}{\ul{\DIFF}^{n-k-1}_{r,h}\ul{\omega}_h}
+ \norm{L^2\Lambda^{k+1}(\Omega)}{\ltproj{r+1}{k+1}{h}(\DIFF\alpha)-\DIFF\alpha} \opn{h}{\ul{\omega}_h}.
\end{align*}
For each $f\in\FM{n}(\Mh)$ we have $\PL{r}{\ell}(f)\subset \PLtrim{r+1}{\ell}(f)$ so the projector $\ltproj{r+1}{\ell}{f}$ on $\PLtrim{r+1}{\ell}(f)$ has approximation properties in $L^2\Lambda^\ell(f)$ at least as good as those of the $L^2$-orthogonal projector on $\PL{r}{\ell}(f)$. Invoking these approximation properties (see \cite[Theorem 1.45]{Di-Pietro.Droniou:20}) concludes the proof of \eqref{eq:adjoint.consistency}.
\end{proof}

We conclude this section by proving the adjoint consistency estimates involving the discrete inner products.

\begin{proof}[Proof of Corollary \ref{cor:adjoint.consistency.inner}]
  By Sobolev embedding, since $2s>\dtop$ and $\zeta\in H^s\Lambda^k(\Omega)$ we have $\zeta\in C^0\Lambda^k(\overline{\Omega})$, so $\uI{k}{h}\zeta$ is well-defined and $\star\zeta$ satisfies the assumptions of Theorem \ref{thm:adjoint.consistency} (with $k\leftarrow n-k$; notice that the condition $k\ge 1$ in Corollary \ref{cor:adjoint.consistency.inner} translates into $n-k\le n-1$ after this substitution).

  Let $\mathfrak E$ be the left-hand side of \eqref{eq:adjoint.consistency.inner}. The consistency \eqref{eq:inner.product.consistency} of the local inner products and a Cauchy--Schwarz inequality show that
  \[
  \left|(\uI{k}{h}\zeta,\ul{\DIFF}^{k-1}_{r,h}\ul{\mu}_h)_{k,h}-\int_\Omega \zeta\wedge\star P^k_{r,h}\ul{\DIFF}^{k-1}_{r,h}\ul{\mu}_h\right|\lesssim
  h^{r+1}\seminorm{H^{(r+1,s)}(\Mh)}{\zeta}\opn{h}{\ul{\DIFF}^{k-1}_{r,h}\ul{\mu}_h},
  \]
  and thus, since $P^k_{r,h}\ul{\DIFF}^{k-1}_{r,h}=\DIFF^{k-1}_{r,h}$ by \cite[Eq.~(3.31)]{Bonaldi.Di-Pietro.ea:24},
  \begin{equation}\label{eq:adjoint.consistency.inner.1}
    \mathfrak E\lesssim \left| \int_\Omega \zeta\wedge\star \DIFF^{k-1}_{r,h}\ul{\mu}_h - \int_\Omega \delta\zeta\wedge \star P^{k-1}_{r,h}\ul{\mu}_h\right| + 
    h^{r+1}\seminorm{H^{(r+1,s)}(\Mh)}{\zeta}\opn{h}{\ul{\DIFF}^{k-1}_{r,h}\ul{\mu}_h}.
  \end{equation}
  If $a$ is a $k$-form and $b$ is an $(n-k)$ form, we have $a\wedge\star b=(-1)^{k(n-k)}\star a\wedge b$.
  Applying this relation to each integrand in \eqref{eq:adjoint.consistency.inner.1} and recalling that $\delta=(-1)^k\star^{-1}\DIFF\star$, we infer
  \begin{multline*}
    \mathfrak E\lesssim \left| (-1)^{k(n-k)}\int_\Omega \star\zeta\wedge \DIFF^{k-1}_{r,h}\ul{\mu}_h - (-1)^{(k-1)(n-k+1)}(-1)^k\int_\Omega \DIFF(\star\zeta)\wedge  P^{k-1}_{r,h}\ul{\mu}_h\right| \\
    + 
    h^{r+1}\seminorm{H^{(r+1,s)}(\Mh)}{\zeta}\opn{h}{\ul{\DIFF}^{k-1}_{r,h}\ul{\mu}_h}.
  \end{multline*}
  We conclude the proof by invoking the adjoint consistency estimate \eqref{eq:adjoint.consistency} with $k\leftarrow n-k$, $\ul{\omega}_h=\ul{\mu}_h$ and $\alpha=\star\zeta$, by noticing that $(-1)^{k(n-k)}\times (-1)^{(k+1)(n-k-1)}(-1)^k=(-1)^{n-k+1}$ and by recalling that $\star$ is an isometry.
\end{proof}

\section{Construction of the conforming lifting}\label{sec:construction.lifting}

\subsection{Cochain complexes: preliminaries and local problem}\label{sec:cochains}

Let us start by recalling some notions on chains and cochains.
On the simplicial submesh $\Sh(f)$ of a $d$-cell $f\in\Mh$, we form for $k\in\{0,\ldots,d\}$ the usual (simplicial) \emph{chain space} $\chain{k}(\Sh(f))$ by declaring a \emph{$k$-chain} to be any finite linear combination
\begin{equation}\label{eq:def.chain}
  w = \sum_{F \in \FM{k}(\Sh(f))} w_F\, F\quad \text{ with $w_F\in\Real$ for all $F\in\FM{k}(\Sh(f))$}.
\end{equation}
The \emph{support} of $w$ is the set of simplices $F$ for which $w_F \neq 0$.
The same construction is done, for $k\in\{0,\ldots,d-1\}$, on the boundary submesh $\Sh(\partial f)$.

We endow the chain spaces with the \emph{boundary operator} $\bd{k} : \chain{k}(\Sh(f)) \to \chain{k-1}(\Sh(f))$, defined on each $k$-simplex $F$ by summing its oriented $(k-1)$-simplices:
\[
\bd{k} F = \sum_{F' \in\FM{k-1}(\Sh(F))} \epsilon_{F,F'}\, F',
\]
where $\epsilon_{F,F'} = \pm 1$ depending if the orientation induced on $F'$ by $F$ matches the intrinsic orientation of $F'$ or not.

Within each chain space $\chain{k}(\Sh(f))$, we define the subspaces of \emph{$k$-cycles} as $\cycle{k}(\Sh(f)) \coloneqq \Ker \bd{k}$, and of \emph{$k$-boundaries} as $\boundary{k}(\Sh(f)) \coloneqq \Image \bd{k-1}$.
Since $\bd{k-1} \circ \bd{k} = 0$ (chain complex property), it follows that $\boundary{k}(\Sh(f)) \subset \cycle{k}(\Sh(f))$.
The $k$-th \emph{homology space} $H_k(\Sh(f))$ is then defined as the quotient
\[
H_k(\Sh(f)) \coloneqq \cycle{k}(\Sh(f)) / \boundary{k}(\Sh(f)).
\]
We briefly recall the standard notion of \emph{reduced homology} (see, e.g., \cite{Allen:02}), applied in the specific case of the (connected) simplicial mesh $\Sh(f)$.  
Fix a 0-simplex $F^* \in \FM{0}(\Sh(f))$, and define $\tilde C_0(\Sh(f))\coloneq \mathrm{span}\{F-F^*\,:\,F\in\FM{0}(\Sh(f))\backslash\{F^*\}\}$ (we also have $\tilde C_0(\Sh(f)) \approx \chain{0}(\Sh(f)) / \mathrm{span}\{ F^* \}$, this second space identifying $0$-chains differing by a multiple of $F^*$). For all $k > 0$, we set $\tilde C_k(\Sh(f)) \coloneqq \chain{k}(\Sh(f))$.
Notice that the boundary operator $\partial_1$ is compatible with this quotient since the boundary of any $1$-simplex $F$ is
$\partial_1 F = \epsilon_{F,F_1'}(F_1' - F^*) + \epsilon_{F,F_2'}(F_2' - F^*) \in \tilde C_0(\Sh(f))$, where $F_1', F_2' \in \FM{0}(\Sh(f))$ are the $0$-simplices of $F$ and $\epsilon_{F,F_1'} + \epsilon_{F,F_2'} = 0$.  
By definition, the $k$-th homology space of the modified chain complex is the \emph{reduced} $k$-th homology space $\tilde H_k(\Sh(f))$.
For $k > 0$, this agrees with the usual homology $H_k(\Sh(f))$, while $\tilde H_0(\Sh(f)) = \{0\}$ since $\Sh(f)$ is connected.

A \emph{$k$-cochain} on $\Sh(f)$ is a linear functional $\tilde \lambda : \chain{k}(\Sh(f)) \to \Real$.
Given a $k$-chain $w \in \chain{k}(\Sh(f))$, we denote $\tilde \lambda(w) \coloneqq \langle \tilde \lambda, w \rangle$.
The space of all $k$-cochains is denoted $\cochain{k}(\Sh(f))$.
Its canonical basis is formed by the \emph{elementary cochains} $\{\hat F \}_{F \in \FM{k}(\Sh(f))}$, dual to each $k$-simplex, that satisfy $\langle \hat F, F' \rangle = \delta_{F,F'}$.
Accordingly, any $\tilde \lambda \in \cochain{k}(\Sh(f))$ decomposes uniquely as
\[
\tilde \lambda = \sum_{F \in \FM{k}(\Sh(f))} \tilde \lambda_F\, \hat F,
\qquad
\tilde \lambda_F \coloneqq \langle \tilde \lambda, F \rangle,
\]
and we refer to $\tilde \lambda_F$ as the \emph{coefficient} of $\tilde \lambda$ on the $k$-simplex $F$.

The \emph{coboundary operator} $\cobd{k} : \cochain{k}(\Sh(f)) \to \cochain{k+1}(\Sh(f))$ is defined as the adjoint of the boundary: for every $k$-cochain $\tilde \lambda \in \cochain{k}(\Sh(f))$ and any $(k+1)$-chain $w$,
\begin{equation}\label{eq:cobd}
  \langle \cobd{k} \tilde \lambda,\, w \rangle = \langle \tilde \lambda,\, \bd{k+1} w \rangle.
\end{equation}
One verifies that $\cobd{k+1} \circ \cobd{k} = 0$, yielding a cochain complex on $\Sh(f)$.

Finally, any $k$-chain $w \in \chain{k}(\Sh(f))$ restricts to a $k$-chain $\restr{w}{\partial f} \in \chain{k}(\Sh(\partial f))$ simply by discarding, in \eqref{eq:def.chain}, all $k$-simplices that do not belong to $\Sh(\pf)$.
Since, for any $0 \le k \le d-1$, $\chain{k}(\Sh(\pf))\subset\chain{k}(\Sh(f))$, any $k$-cochain $\tilde \lambda \in \cochain{k}(\Sh(f))$ admits a \emph{trace} $\ctr^k_{\partial f} \tilde \lambda \in \cochain{k}(\Sh(\partial f))$, defined by restricting the domain to $\chain{k}(\Sh(\partial f))$.

\medskip

We now state and prove an existence result for a local problem in cochain spaces, 
building on the theory developed in \cite{Pitassi.Ghiloni.ea:22}. 
This problem will be essential to solving, in trimmed polynomial spaces, the local problems at the core of the construction of the lifting.

\begin{lemma}[Local boundary value problem for cochain maps]
  \label{lem:cochain.problem}
  Let $f \in \FM{d}(\Mh)$ be a $d$-cell with $d \geq 1$ and $k\in\{0,\ldots,d\}$ be a cochain degree.
  Let $\tilde{\xi} \in \cochain{k+1}(\Sh(f))$ be such that $\cobd{k+1} \tilde{\xi} = 0$, and $\tilde{\theta} \in \cochain{k}(\Sh(\partial f))$ be such that the following compatibility condition holds:
  \begin{subequations}\label{eq:local.pb.cochain.cond}
    \begin{align}
      \label{eq:cochain.condition.d=k+1}
      \text{If $d=k+1$}:&\qquad \sum_{F\in\FM{k+1}(\Sh(f))}\tilde{\xi}_F = \sum_{F'\in\FM{k}(\Sh(\partial f))}\tilde{\theta}_{F'},\\
      \label{eq:cochain.condition.d>=k+2}
      \text{If $d\ge k+2$}:&\qquad \ctr^{k+1}_{\partial f} \tilde{\xi}=\cobd{k} \tilde{\theta},
    \end{align}
  \end{subequations}
  where each $F\in\FM{k+1}(\Sh(f))$ and each $F'\in\FM{k}(\Sh(\partial f))$ appearing in the sums in condition \eqref{eq:cochain.condition.d=k+1} carries the orientation induced by $f$ (which is possible since $f$ has dimension $d=k+1$ in this case).
  Then, there exists $\tilde{\lambda} \in \cochain{k}(\Sh(f))$ such that
  \begin{subequations} \label{eq:cochain.problem}
    \begin{empheq}[left=\empheqlbrace]{align}
      \cobd{k} \tilde{\lambda} &= \tilde{\xi},  \label{eq:cochain.problem.internal} \\
      \ctr^k_{\partial f} \tilde{\lambda} &= \tilde{\theta}, \label{eq:cochain.problem.boundary}
    \end{empheq}
  \end{subequations}
  and
  \begin{equation}\label{eq:cochain.problem.estimates}
    \sum_{F\in\FM{k}(\Sh(f))}\tilde{\lambda}_F^2\lesssim \sum_{F\in\FM{k+1}(\Sh(f))}\tilde{\xi}_F^2
    +\sum_{F\in\FM{k}(\Sh(\partial f))}\tilde{\theta}_{F}^2.
  \end{equation}
\end{lemma}

\begin{proof}
  The proof is split in three steps. In the first one, we construct $\tilde \lambda$ explicitly by leveraging the theory developed in \cite{Di-Pietro.Droniou.ea:25}. The second and third step, respectively, check that the constructed $\tilde{\lambda}$ satisfies \eqref{eq:cochain.problem} and \eqref{eq:cochain.problem.estimates}.

  \smallskip
  \emph{Step 1: construction of $\tilde{\lambda}$}.
  Applying the algorithm in \cite[Appendix C]{Di-Pietro.Droniou.ea:25} yields a complement space $\cspace{k}(\Sh(f))$ for the boundary cycles, that is, a space such that
   \begin{equation}
    \cycle{k}(\Sh(f)) = \cycle{k}(\Sh(\partial f)) \oplus \cspace{k}(\Sh(f)),
    \label{eq:relative.decomposition}
  \end{equation}
  together with a basis $\mc B_k$ of $\cspace{k}(\Sh(f))$ and a dual family of $k$-simplices $\mc F_k\coloneq\{F_z\,:\,z\in\mc B_k\}$ of $\Sh(f)$, such that each $z = \sum_{F \in \FM{k}(\Sh(f))} z_F\,F\in\mc B_k$ satisfies the bound
  \begin{equation}\label{eq:bound.z}
    \sum_{F\in\FM{k}(\Sh(f))}z_F^2\lesssim 1
  \end{equation}
  and
  \begin{equation}\label{eq:prop.Fz}
    \begin{aligned}
      F_z \not\subset \partial f &\quad \forall F_z \in \mc F_k,\\
      \inner{F_z}{z'} = \delta_{z,z'}&\quad \forall z,z' \in \mc B_k.
    \end{aligned}
  \end{equation}
For $k = 0$, recall that $\cycle{0}(\Sh(f)) = \boundary{0}(\Sh(f)) \oplus \mathrm{span}\{F^*\}$, as in \cite[Eq.~(27)]{Di-Pietro.Droniou.ea:25}.  
Since $f$ is contractible, its reduced homology is trivial, and therefore all homology groups of $f$ vanish. Consequently, for each $k$-cycle $z \in \mc B_k$, there exists a $(k+1)$-chain $w_z \in \chain{k+1}(\Sh(f))$ such that
  \begin{equation}\label{eq:def.wz}
    \bd{k+1} w_z = z.
  \end{equation}
  Moreover, by \cite[Appendix C]{Di-Pietro.Droniou.ea:25} we can select $w_z$ such that, writing $w_z=\sum_{F' \in \FM{k+1}(\Sh(f))} w_{z,F'}\,F'$, the following bound holds:
  \begin{equation}\label{eq:bound.w}
  \sum_{F'\in\FM{k+1}(\Sh(f))} w_{z,F'}^2\lesssim 1.
  \end{equation}
  We now define $\tilde \lambda$ by setting its value on each $k$-simplex, on the boundary or the interior of $f$:
  \begin{enumerate}[label=\arabic*)]
  \item On the boundary: Set
    \begin{equation}\label{eq:t.lambda.boundary}
      \inner{\tilde{\lambda}}{F} = \inner{\tilde{\theta}}{F}
      \qquad \forall F\in\FM{k}(\Sh(\pf)).
    \end{equation}
  \item In the interior: For each $k$-simplex $F \in \FM{k}(\Sh(f))$ such that $F \not\subset \partial f$, set 
    \begin{equation}
      \inner{\tilde \lambda}{F} \coloneqq
      \begin{cases}
        \inner{\tilde \xi}{w_z} - \inner{\tilde \theta}{{\restr{z}{\partial f}}} &\text{if } F = F_z  \in \mc F_k, \\
        0 &\text{otherwise}.
      \end{cases}
      \label{eq:formula}
    \end{equation}
  \end{enumerate}

  \smallskip
  \emph{Step 2: verification of \eqref{eq:cochain.problem}}.
  The boundary equation \eqref{eq:cochain.problem.boundary} is trivially satisfied by construction of $\tilde{\lambda}$ on the boundary, so we only have to check the interior relation \eqref{eq:cochain.problem.internal}. 

  To this purpose, we first decompose $\chain{k+1}(\Sh(f))$. Let $\mc B_k^*$ be a basis of $\cycle{k}(\Sh(\partial f))$; since $f$ is contractible, for each $z\in\mc B_k^*\subset \cycle{k}(\Sh(f))$, there exists $w_z\in\chain{k+1}(\Sh(f))$ such that \eqref{eq:def.wz} holds. By \eqref{eq:relative.decomposition}, $\mc B_k\cup\mc B_k^*$ is a basis of $\cycle{k}(\Sh(f))=\Image\bd{k+1}$ (using again that $f$ is contractible), so the first isomorphism theorem applied to $\bd{k+1}$ yields
  \begin{equation*} 
    \chain{k+1}(\Sh(f))=\cycle{k+1}(\Sh(f))\oplus\mathrm{span}\{w_z\,:\,z\in\mc B_k\}\oplus\mathrm{span}\{w_z\,:\,z\in\mc B_k^*\}.
  \end{equation*}
  We will prove \eqref{eq:cochain.problem.internal} by showing that it holds when tested against chains in each space of this decomposition.

  \emph{Step 2a: verification of \eqref{eq:cochain.problem.internal} against $w\in\cycle{k+1}(\Sh(f))$}.
  If $w\in \cycle{k+1}(\Sh(f))$, then $\bd{k+1}w=0$ by definition and $w=\bd{k+2}t$ for some $t\in\chain{k+2}(\Sh(f))$, so
  \[
  \inner{\cobd{k}\tilde{\lambda}}{w}
  \overset{\eqref{eq:cobd}}=\inner{\tilde{\lambda}}{\bd{k+1}w}=0
  \quad\text{ and }\quad
  \inner{\tilde{\xi}}{w}=\inner{\tilde{\xi}}{\bd{k+2}t}=\inner{\cobd{k+1}\tilde{\xi}}{t}=0,
  \]
  which shows that $\inner{\cobd{k}\tilde{\lambda}}{w}=\inner{\tilde{\xi}}{w}$ as required.

  \emph{Step 2b: verification of \eqref{eq:cochain.problem.internal} against $w_z$ for $z\in\mc B_k$}.
  Since $\tilde \lambda$ vanishes on every interior $k$-simplex except those in $\mc F_k$ (see \eqref{eq:formula}) and $z$ has only $F_z$ (with coefficient 1) in its support as simplex of $\mc F_k$ (see \eqref{eq:prop.Fz}), we have $\inner{\tilde \lambda}{z - \restr{z}{\partial f} - F_z} = 0$.
  Hence, $\inner{\tilde \theta}{\restr{z}{\partial f}} \overset{\eqref{eq:t.lambda.boundary}}= \inner{\tilde \lambda}{\restr{z}{\partial f}} = \inner{\tilde \lambda}{z-F_z}$ and thus
  \[
  \inner{\tilde \lambda}{F_z}
  \overset{\eqref{eq:formula}}=
  \inner{\tilde \xi}{w_z} - \inner{\tilde \theta}{\restr{z}{\partial f}}
  \overset{\eqref{eq:cochain.problem.boundary}}=\inner{\tilde \xi}{w_z} - \inner{\tilde \lambda}{z - F_z}.
  \]
  We infer that
  \begin{equation*}
    \inner{\cobd{k}\tilde \lambda}{w_z}
    \overset{\eqref{eq:cobd}}=\inner{\tilde \lambda}{\bd{k+1}w_z}
    \overset{\eqref{eq:def.wz}}=\inner{\tilde \lambda}{z}
    =\inner{\tilde \lambda}{z -F_z}+\inner{\tilde \lambda}{F_z}
    =\inner{\tilde \xi}{w_z}.
  \end{equation*}

  \emph{Step 2c: verification of \eqref{eq:cochain.problem.internal} against $w_z$ for $z\in\mc B_k^*$}.
  We distinguish two cases depending on the dimension $d$ of $f$.
  If $d = k+1$, the space of $k$-cycles $Z_{k}(\Sh(\partial f))$ is $1$-dimensional (c.f.~\cite[Example~2.17]{Allen:02}), spanned by $z$ such that $w_z$ is the $(k+1)$-chain
  \begin{equation}\label{eq:cycle.d=k+1.wz}
    w_z = \sum_{F\in \FM{k+1}(\Sh(f))}F.
  \end{equation}
  In $\bd{k+1}w_z$, each internal $k$-simplex $F'\not\subset\pf$ appears in the boundary of two $(k+1)$-simplices in $\FM{k+1}(\Sh(f))$ with opposite orientations, so
  \begin{equation}\label{eq:cycle.d=k+1.z}
    z=\bd{k+1} w_z=\sum_{F'\in\FM{k}(\Sh(\pf))}F'.
  \end{equation}
  The proof is then concluded by writing
  \[
  \inner{\cobd{k}\tilde \lambda}{w_z} \overset{\eqref{eq:def.wz}}
  =\inner{\tilde \lambda}{z} \overset{\eqref{eq:t.lambda.boundary},\eqref{eq:cycle.d=k+1.z}}=\sum_{F'\in\FM{k}(\Sh(\pf))}\tilde{\theta}_{F'}\overset{\eqref{eq:cochain.condition.d=k+1}}= \sum_{F\in\FM{k+1}(\Sh(f))}\tilde{\xi}_F
  \overset{\eqref{eq:cycle.d=k+1.wz}}=\inner{\tilde \xi}{w_z}.
  \]

  Let us now assume that $d \ge k + 2$. Since $f$ is homeomorphic to a $d$-dimensional ball, $\pf$ is homeomorphic to a $(d-1)$-dimensional sphere; the reduced homology group $\tilde H_k(\Sh(\partial f))$ is thus isomorphic to the reduced homology group of the $(d-1)$-sphere, and is therefore zero (as $k<d-1$). As a consequence, the chain $w_z$ previously selected in $\chain{k+1}(\Sh(f))$ can actually be chosen in the space $\chain{k+1}(\Sh(\pf))$ of chains on the boundary. We then conclude that \eqref{eq:cochain.problem.internal} holds on $w_z$ by writing
  \[
  \inner{\tilde{\xi}}{w_z}
  = \inner{\ctr^k_{\partial f}\tilde{\xi}}{w_z}
  \overset{\eqref{eq:cochain.condition.d>=k+2}}=
  \inner{\cobd{k}\tilde{\theta}}{w_z}
  \overset{\eqref{eq:t.lambda.boundary}}=
  \inner{\cobd{k}\tilde{\lambda}}{w_z}.
  \]

  \smallskip
  \emph{Verification of \eqref{eq:cochain.problem.estimates}.}
  We split the sum in the left-hand side of \eqref{eq:cochain.problem.estimates} in three blocs and use the definition of $\tilde{\lambda}$ to write
  \begin{align}
    \sum_{F\in\FM{k}(\Sh(f))}\tilde{\lambda}_F^2={}&
    \sum_{F\in\FM{k}(\Sh(\pf))}\tilde{\lambda}_F^2+
    \sum_{F\in\FM{k}(\Sh(f)),F\not\subset\pf,F\not\in\mc F_k}\tilde{\lambda}_F^2\nonumber\\
    &+\sum_{F\in\FM{k}(\Sh(f)),F=F_z\in\mc F_k}\tilde{\lambda}_F^2\nonumber\\
    ={}&
    \sum_{F\in\FM{k}(\Sh(\pf))}\tilde{\theta}_F^2
    +\sum_{z\in\mc B_k}\left(\inner{\tilde \xi}{w_z} - \inner{\tilde \theta}{\restr{z}{\partial f}}\right)^2.
    \label{eq:estim.cochain.1}
  \end{align}
  For the term in the last sum, we use the bounds \eqref{eq:bound.z} and \eqref{eq:bound.w} together with Cauchy--Schwarz inequalities to see that \cite[Eq.~(50)]{Di-Pietro.Droniou.ea:25}
  \[
  \left(\inner{\tilde \xi}{w_z} - \inner{\tilde \theta}{\restr{z}{\partial f}}\right)^2
  \lesssim \sum_{F\in\FM{k+1}(\Sh(f))}\tilde{\xi}_F^2 + \sum_{F\in\FM{k}(\Sh(\pf))}\tilde{\theta}_F^2.
  \]
  Summing this over $z\in\mc F_k$, noting that $\card(\mc F_k)\lesssim 1$ and plugging the resulting estimate into \eqref{eq:estim.cochain.1} concludes the proof of \eqref{eq:cochain.problem.estimates}.
\end{proof}

\subsection{Local problems in trimmed polynomial spaces}\label{sec:local.problems}

This section presents two existence results to local problems in spaces of finite element forms (trimmed polynomial spaces) that will form the foundations of the lifting operator.
Given a $d$-cell $f$ and a polynomial degree $s\ge 1$, we denote by
\[
\PLtrim{s}{k}(\Sh(f))\coloneq\{\omega\in H\Lambda^k(f)\,:\,\restr{\omega}{F}\in\PLtrim{s}{k}(F)\quad\forall F\in\FM{d}(\Sh(f))\}
\]
the space of conforming trimmed polynomial $k$-forms of degree $s$ on the mesh $\Sh(f)$. We space $\PLtrim{s}{k}(\Sh(\partial f))$ denotes the restriction to $\Sh(\partial f)$ of these forms, and coincides with forms whose restriction to each $f'\in\FM{d-1}(f)$ belong to $\PLtrim{s}{k}(\Sh(f'))$ and with single-valued traces on $f'\cap f''$ for all $f',f''\in\FM{d-1}(f)$.
If $\mu\in \PLtrim{s}{k}(\Sh(f))$, with an abuse of notation we write $\tr_{\partial f}\mu$ and $\DIFF \tr_{\partial f}\mu$ for the functions on $\partial f$ such that, for all $f'\in\FM{d-1}(f)$, $(\tr_{\partial f}\mu)|_{f'}=\tr_{f'}\mu$ and $(\DIFF\tr_{\partial f}\mu)|_{f'}=\DIFF \tr_{f'}\mu$.

\begin{lemma}[Local boundary value problem in finite element spaces]
  \label{lem:first.local.problem}
  Let $f \in \FM{d}(\Mh)$ be a $d$-cell with $d \geq 1$, $k\in\{0,\ldots,d\}$ be a form degree, and $s \ge 1$ be a polynomial degree.
  Let $\xi \in \PLtrim{s}{k+1}(\Sh(f))$ be such that $\DIFF \xi = 0$, and $\theta \in \PLtrim{s+1}{k}(\Sh(\partial f))$ be such that the following compatibility condition holds:
  \begin{subequations}\label{eq:local.pb.cond}
    \begin{align}
      \label{eq:condition.d=k+1}
      \text{If $d=k+1$}:&\qquad \int_{f} \xi = \int_{\partial f} \theta,\\
      \label{eq:condition.d>=k+2}
      \text{If $d\ge k+2$}:&\qquad \tr_{\partial f} \xi = \DIFF \theta.
    \end{align}
  \end{subequations}
  Then, there exists $\lambda \in \PLtrim{s+1}{k}(\Sh(f))$ such that
  \begin{subequations} \label{eq:local.pb}
    \begin{empheq}[left=\empheqlbrace]{align}
      \DIFF \lambda &= \xi,  \label{eq:local.pb.diff} \\
      \tr_{\partial f} \lambda &= \theta \label{eq:local.pb.boundary}
    \end{empheq}
  \end{subequations}
  and
  \begin{equation}\label{eq:local.pb.estimate}
    \norm{f}{\lambda}^2
    \lesssim h_f \norm{\partial f}{\theta}^2
    + h_f^2 \norm{f}{\xi}^2.
  \end{equation}
\end{lemma}

\begin{proof}
  We construct $\lambda$ by setting its degrees of freedom (DOFs), and the equations in \eqref{eq:local.pb} are checked by verifying the equality of the DOFs on the left- and right-hand side of each equation. We recall that the degrees of freedom of $\zeta\in\PLtrim{s}{\ell}(\Sh(f))$ are \cite{Arnold:18}
  \[
  \int_{F}\tr_F \zeta\wedge\mu\qquad \forall \mu\in \PL{s+\ell-d'-1}{d'-\ell}(F)\quad \forall F\in\FM{d'}(\Sh(f))\,,\quad\forall d'\ge \ell.
  \]
  The proof is split in three steps. First, we reduce the problem to forms that have a zero trace integral on the lowest-dimensional mesh entities. Then, we build the degrees of freedom for the reduced problem, and check \eqref{eq:local.pb}, by an inductive process on the dimension of the cells these DOFs are attached to. Finally, we check the estimate \eqref{eq:local.pb.estimate}.

  \medskip

  \noindent\emph{Step 1: reduction to zero-average forms}.
  For all $s\ge 1$, we denote by $R_{\Sh(f)}^*:\PLtrim{s}{*}(\Sh(f))\to \cochain{*}(\Sh(f))$ the graded de Rham map, defined by
  \begin{equation}\label{eq:de.rham.map}
    R_{\Sh(f)}^k\omega = \left(\int_F \tr_F\omega\right)_{F\in\FM{k}(\Sh(f))}\qquad\forall k\in\{0,\ldots,d\},\quad\forall \omega\in\PLtrim{s}{k}(\Sh(f)).
  \end{equation}
  Similarly, $R_{\Sh(\partial f)}^*$ is the de Rham graded map on $\Sh(\partial f)$ and we recall that these maps commute with the trace operator as well as the exterior derivative: $\ctr^k_{\partial f}\circ R_{\Sh(f)}^k=R_{\Sh(\partial f)}^k \circ \tr^k_{\partial f}$ and $\cobd{k}\circ R_{\Sh(f)}^k=R_{\Sh(f)}^{k+1}\circ \DIFF$. The inverse of the de Rham maps on $\PLtrim{1}{*}(\Sh(f))$ are the Whitney maps $W_{\Sh(f)}^*:\cochain{*}(\Sh(f))\to \PLtrim{1}{*}(\Sh(f))$ defined by
  \[
  W_{\Sh(f)}^k\tilde{\omega} = \sum_{F\in\FM{k}(\Sh(f))}\tilde{\omega}_F \Wbf_{F}^k\qquad\forall k\in\{0,\ldots,d\},\quad\forall\tilde{\omega}\in\cochain{k}(\Sh(f)),
  \]
  where $\Wbf_F^k$ is the Whitney form associated with $F$. The Whitney maps also commute with the trace and the exterior derivative.

  Let $\tilde{\xi}=R_{\Sh(f)}^{k+1}\xi$ and $\tilde{\theta}=R_{\Sh(\partial f)}^k\theta$. By commutation of the de Rham map with the trace and the exterior derivative, the condition $\DIFF\xi=0$ gives $\cobd{k+1}\tilde{\xi}=0$, and \eqref{eq:local.pb.cond} implies \eqref{eq:local.pb.cochain.cond}. There exists therefore $\tilde{\lambda}\in \cochain{k}(\Sh(f))$ solution of \eqref{eq:cochain.problem} that satisfies \eqref{eq:cochain.problem.estimates}.

  Instead of directly constructing $\lambda$ solution of \eqref{eq:local.pb}--\eqref{eq:local.pb.estimate}, we will construct $\wh{\lambda}:=\lambda-W_{\Sh(f)}^k\tilde{\lambda}\in\PLtrim{s+1}{k}(\Sh(f))$. Using
  again the commutation properties of the de Rham and Whitney maps, it is straightforward to see that $\lambda$ solves \eqref{eq:local.pb} if and only if $\wh{\lambda}$ solves
  \begin{subequations} \label{eq:local.pb.average}
    \begin{empheq}[left=\empheqlbrace]{align}
      \DIFF \wh{\lambda} &= \wh{\xi},  \label{eq:local.pb.average.diff} \\
      \tr_{\partial f} \wh{\lambda} &= \wh{\theta}, \label{eq:local.pb.average.boundary}
    \end{empheq}
  \end{subequations}
  where
  \[
  \wh{\xi}\coloneq \xi-W_{\Sh(f)}^{k+1}\tilde{\xi}\in\PLtrim{s}{k+1}(\Sh(f))\quad\text{ and }\quad
  \wh{\theta}\coloneq \theta-W_{\Sh(\partial f)}^k\tilde{\theta}\in\PLtrim{s+1}{k}(\Sh(\partial f)).
  \]
  The interest of \eqref{eq:local.pb.average} over \eqref{eq:local.pb} is that its data verify zero average conditions: since $R_{\Sh(f)}^{k+1}\circ W_{\Sh(f)}^{k+1}=\mathrm{Id}_{\cochain{k+1}(\Sh(f))}$ (and similarly for the boundary maps), we have
  \begin{equation}\label{eq:zero.average}
    R_{\Sh(f)}^{k+1}\wh{\xi}=0\quad\text{ and }\quad R_{\Sh(\partial f)}^k \wh{\theta}=0.
  \end{equation}
  Moreover, the same commutation properties as before show that 
  \begin{equation}\label{eq:diff.wh.xi}
    \DIFF\wh{\xi}=0
  \end{equation}
  and that \eqref{eq:condition.d>=k+2} still holds for $(\wh{\xi},\wh{\theta})$:
  \begin{equation}\label{eq:condition.d>=k+2.hat}
    \text{If $d\ge k+2$ then $\tr_{\partial f} \wh{\xi} = \DIFF \wh{\theta}$}.
  \end{equation}

  \medskip

  \noindent\emph{Step 2: construction of a solution to \eqref{eq:local.pb.average}}.
  We set the DOFs of $\wh{\lambda}$ on each $F\in\Delta_{d'}(\Sh(f))$ by induction on $d'\in\{k,\ldots,d\}$.

  \underline{Base case $d'=k$}. Let $F\in\FM{k}(\Sh(f))$. If $F\subset \partial f$, then we set the DOFs of $\wh{\lambda}$ on $F$ equal to those of $\wh{\theta}$. If $F\not\subset \partial f$, we simply set the DOFs of $\wh{\lambda}$ on $F$ to be zero.
  All the DOFs on $k$-simplices of $\wh{\lambda}$ are thus set in a way that ensures that the DOFs on those simplices on each side of \eqref{eq:local.pb.average.boundary} match.
  Notice that there are no DOFs for the equality \eqref{eq:local.pb.average.diff} between $(k+1)$-forms on the $k$-simplex $F$.

  By the second relation in \eqref{eq:zero.average}, this construction gives 
  \begin{equation}\label{eq:Rh.whlambda.zero}
    \int_F\tr_F\wh{\lambda}=0\qquad\forall F\in\FM{k}(\Sh(f)).
  \end{equation}

  \underline{Inductive step}. Let $d'\in \{k+1,\ldots,d\}$ and let us assume that all the DOFs of $\wh{\lambda}$ on simplices of dimension $\le d'-1$ have been set in a such a way that the associated DOFs on both sides of each equation of \eqref{eq:local.pb.average} match.

  Let $F\in\Delta_{d'}(\Sh(f))$ and assume first that $F\subset \partial f$. As before, we then impose the equality of the DOFs on $F$ on each side of \eqref{eq:local.pb.average.boundary}, which sets the DOFs of $\wh{\lambda}$ on $F$.
  Combining this choice with the induction hypothesis, all DOFs of $\tr_F\wh{\lambda}\in\PLtrim{s+1}{k}(F)$ and $\tr_F\wh{\theta}\in\PLtrim{s+1}{k}(F)$ on $F$ and its boundary simplices (of any dimension) match and, by unisolvence, we infer that
  \[
  \tr_{F}\wh{\lambda}=\tr_{F}\wh{\theta}.
  \]
  Applying the exterior derivative to this relation and accounting for its commutativity with the trace operator gives $\tr_F(\DIFF\wh{\lambda})=\tr_F(\DIFF\wh{\theta})=\tr_F\wh{\xi}$, where the last equality comes from the compatibility condition \eqref{eq:condition.d>=k+2.hat} (which is valid here since, if a $d'$-simplex $F$ is contained in $\partial f$, then $d>d'\ge k+1$). This shows that the DOFs on $F$ on each side of \eqref{eq:local.pb.average.diff} match, and concludes the check of \eqref{eq:local.pb.average} on the simplex $F$.

  Assume now that $F\in\Delta_{d'}(\Sh(f))$ is not contained in $\partial f$. We need to fix the value of $\int_F \tr_F\wh{\lambda}\wedge\mu$ for all $\mu\in \PL{s+k-d'}{d'-k}(F)$.
  To this end, we will use the following isomorphism (see \cite[Eq.~(2.17)]{Bonaldi.Di-Pietro.ea:24} and notice that $d'-k\ge 1$):
  \begin{equation}\label{eq:decomposition.PL}
    \begin{array}{rl}
      \KOSZUL_F\PL{s+k-d'}{d'-k}(F)\times \KOSZUL_F\PL{s+k-d'-1}{d'+1-k}(F)\to{}& \PL{s+k-d'}{d'-k}(F)\\
      (\omega,\zeta)\mapsto{}&\DIFF\omega+\zeta,
    \end{array}
  \end{equation}
  where $\KOSZUL_F$ is the Koszul operator on $F$.
  Writing, according to this isomorphism, $\mu= \DIFF\omega+\zeta$, we then set
  \begin{equation}\label{eq:def.DOF.lambda}
    \int_F \tr_F\wh{\lambda}\wedge \mu = (-1)^{k+1}\int_F \tr_F\wh{\xi}\wedge\omega+(-1)^k\int_{\partial F}\tr_{\partial F}\wh{\lambda}\wedge \tr_{\partial F}\omega. 
  \end{equation}
  By induction, all the DOFs of $(\tr_{F'}\wh{\lambda})_{F'\in\Delta_{d'-1}(F)}$ have been fixed and $\tr_{\partial F}\wh{\lambda}$ is thus fully determined. It remains to show that this choice ensures that the DOFs of both sides of \eqref{eq:local.pb.average.diff} on $F$ match, that is: for all $\nu\in \PL{s+k-d'}{d'-k-1}(F)$,
  \begin{equation}\label{eq:check.DOF.diff}
    \int_F \tr_F\wh{\xi}\wedge\nu = \int_F \DIFF(\tr_F\wh{\lambda})\wedge\nu
    =(-1)^{k+1}\int_F \tr_F\wh{\lambda}\wedge\DIFF\nu+\int_{\partial F}\tr_{\partial F}\wh{\lambda}\wedge\tr_{\partial F}\nu,
  \end{equation}
  the second equality following from Stokes' formula. We consider two cases.

  \emph{Case $d'=k+1$}. We need to verify \eqref{eq:check.DOF.diff} for all $\nu\in \PL{s-1}{0}(F)=\PL{0}{0}(F)\oplus \KOSZUL_F \PL{s-2}{1}(F)$ (see \eqref{eq:decomp.PL}). If $\nu\in \PL{0}{0}(F)$, then it is a constant scalar and \eqref{eq:check.DOF.diff} reduces to
  \[
  \int_F \tr_F\wh{\xi}=\int_{\partial F}\tr_{\partial F}\wh{\lambda}=\sum_{F'\in\FM{k}(F)}\int_{F'}\tr_{F'}\wh{\lambda},
  \]
  which trivially holds by \eqref{eq:zero.average} (which shows that the left-hand side vanishes) and \eqref{eq:Rh.whlambda.zero} (which shows that the right-hand side vanishes).

  If $\nu\in\KOSZUL_F \PL{s-2}{1}(F)\subset \KOSZUL_F\PL{s+k-d'}{d'-k}(F)$, we can invoke \eqref{eq:def.DOF.lambda} with $(\omega,\zeta)=(\nu,0)$, which precisely gives \eqref{eq:check.DOF.diff}. 

  \emph{Case $d'>k+1$}. By \eqref{eq:decomp.PL} we have 
  \[
  \PL{s+k-d'}{d'-k-1}(F)=\DIFF\PL{s+k-d'+1}{d'-k-2}(F)\oplus \KOSZUL_F \PL{s+k-d'-1}{d'-k}(F)
  \]
  and we therefore need to check \eqref{eq:check.DOF.diff} for $\nu$ in each space in this decomposition. If $\nu=\DIFF\rho$ for some $\rho\in \PL{s+k-d'+1}{d'-k-2}(F)$ then 
  \begin{align*}
    \int_F\tr_F\wh{\xi}\wedge\nu={}&(-1)^k\int_F\cancel{\DIFF(\tr_F\wh{\xi})}\wedge\rho+(-1)^{k+1}\int_{\partial F}\tr_{\partial F}\wh{\xi}\wedge\tr_{\partial F}\rho\\
    ={}&(-1)^{k+1}\int_{\partial F}\tr_{\partial F}(\DIFF\wh{\lambda})\wedge\tr_{\partial F}\rho
  \end{align*}
  where the cancellation in the first equality is justified by $\DIFF(\tr_F\wh{\xi})=\tr_F(\DIFF\wh{\xi})\overset{\eqref{eq:diff.wh.xi}}=0$, while the conclusion follows from $\tr_{\partial F}\wh{\xi}=\tr_{\partial F}(\DIFF\wh{\lambda})$, which holds by induction hypothesis -- which states that, for all $F'\in\FM{d'-1}(F)$, the DOFs on $F'$ and its sub-simplices of the trimmed polynomial forms $\tr_{F'}\wh{\xi}$ and $\tr_{F'}(\DIFF\wh{\lambda})$, and thus the forms themselves, match.
  We continue by writing $\tr_{\partial F}(\DIFF\wh{\lambda})=\DIFF(\tr_{\partial F}\wh{\lambda})$, integrating by parts and using $\partial(\partial F)=0$ (in other words, the boundary terms coming from the integrations by parts on each $F'\in\FM{d'-1}(F)$ cancel out) to get
  \[
  \int_F\tr_F\wh{\xi}\wedge\nu=\int_{\partial F}\tr_{\partial F}\wh{\lambda}\wedge\tr_{\partial F}\DIFF\rho
  =\int_{\partial F}\tr_{\partial F}\wh{\lambda}\wedge\tr_{\partial F}\nu.
  \]
  Since $\DIFF\nu=\DIFF^2\rho=0$, the relation above proves that \eqref{eq:check.DOF.diff} holds when $\nu\in \DIFF\PL{s+k-d'+1}{d'-k-2}(F)$.

  If $\nu\in \KOSZUL_F \PL{s+k-d'-1}{d'-k}(F)\subset\KOSZUL_F \PL{s+k-d'}{d'-k}(F)$ then we can apply \eqref{eq:def.DOF.lambda} with $(\omega,\zeta)=(\nu,0)$ to immediately get \eqref{eq:check.DOF.diff}.

  This concludes the inductive construction of $\wh{\lambda}$ satisfying \eqref{eq:local.pb.average}.

  \medskip

  \emph{Step 3: estimates.}
  Since $\lambda=\wh{\lambda}-W_{\Sh(f)}^k\tilde{\lambda}$, \eqref{eq:local.pb.estimate} follows if we prove that both $\wh{\lambda}$ and $W_{\Sh(f)}^k\tilde{\lambda}$ satisfy this estimate.

  Let us start with $W_{\Sh(f)}^k\tilde{\lambda}$. By \cite[Lemma 19]{Di-Pietro.Droniou.ea:25} and the fact that $\card(\Sh(f))\lesssim 1$, we have
  \begin{align}
    \norm{f}{W_{\Sh(f)}^k \tilde{\lambda}}^2
    &\lesssim h_f^{d-2k} \sum_{F\in\FM{k}(\Sh(f))}\tilde{\lambda}_F^2\nonumber\\
    \overset{\eqref{eq:cochain.problem.estimates}}&\lesssim h_f^{d-2k}\sum_{F\in\FM{k+1}(\Sh(f))}\tilde{\xi}_F^2
    +h_f^{d-2k}\sum_{F\in\FM{k}(\Sh(\partial f))}\tilde{\theta}_F^2.
    \label{eq:local.pb.est.tildelambda}
  \end{align}
  Since $\tilde{\xi}=R_{\Sh(f)}^{k+1}\xi$, the definition \eqref{eq:de.rham.map} of the de Rham map together with a Cauchy--Schwarz inequality and the mesh regularity assumption yield $\tilde{\xi}_F^2\lesssim h_F^{k+1}\norm{F}{\tr_F\xi}^2$ for all $F\in\FM{k+1}(\Sh(f))$.
  Repeatedly using the discrete trace inequality of \cite[Lemma 18]{Di-Pietro.Droniou.ea:25} on the piecewise polynomial form $\xi$ (to go from $F$ of dimension $k+1$ to $f$ of dimension $d$) then leads to 
  \begin{equation}\label{eq:est.tilde.xi}
    \tilde{\xi}_F^2\lesssim h_f^{k+1-(d-(k+1))}\norm{f}{\xi}^2=h_f^{2k+2-d}\norm{f}{\xi}^2.
  \end{equation}
  The same arguments applied to $\theta$ give, for all $F\in\FM{k}(\Sh(\partial f))$,
  \begin{equation}\label{eq:est.tilde.theta}
    \tilde{\theta}_F^2
    \lesssim h_f^{k-(d-1-k)} \norm{\partial f}{\theta}^2
    = h_f^{2k-d+1}\norm{\partial f}{\theta}^2.
  \end{equation}
  Plugging \eqref{eq:est.tilde.xi} and \eqref{eq:est.tilde.theta} into \eqref{eq:local.pb.est.tildelambda} shows that $W_{\Sh(f)}^k\tilde{\lambda}$ satisfies \eqref{eq:local.pb.estimate}.

  We now turn to $\wh{\lambda}$. We first note that, by \cite[Lemma 19]{Di-Pietro.Droniou.ea:25}, \eqref{eq:est.tilde.xi} and \eqref{eq:est.tilde.theta} we have $\norm{f}{\wh{\xi}}\lesssim\norm{f}{\xi}$ and $\norm{\partial f}{\wh{\theta}}\lesssim\norm{\partial f}{\theta}$.
  Hence, it suffices to show that $\wh{\lambda}$ satisfies \eqref{eq:local.pb.estimate} with $\wh{\xi}$, $\wh{\theta}$ in the right-hand side to conclude. For this purpose, we first need to link the norm of $\wh{\lambda}$ with some measure involving only its DOFs. For each $d'\in\{k,\ldots,d\}$ and $F\in\FM{d'}(\Sh(f))$, we define the ``norm of the degrees of freedom of $\wh{\lambda}$ on $F$'' by
  \[
  N_{\mathrm{DOF}}(\wh{\lambda},F)\coloneq \sup_{\mu\in\PL{s+k-d'}{d'-k}(F)\backslash\{0\}}\frac{1}{\norm{F}{\mu}}\int_F \tr_F\wh{\lambda}\wedge\mu
  \]
  (this is also the $L^2$-norm of the $L^2$-projection of $\tr_F\wh{\lambda}$ on the full polynomial space of degree $s+k-d'$). An argument based on reference elements and the fact that $\Sh(f)$ is a regular simplicial mesh with $\lesssim 1$ simplices shows that
  \begin{equation}\label{eq:est.lambda.Ndof}
    \norm{f}{\wh{\lambda}}^2\simeq \sum_{d'=k}^d h_f^{d-d'}\sum_{F\in\FM{d'}(\Sh(f))}N_{\mathrm{DOF}}(\wh{\lambda},F)^2.
  \end{equation}
  Assume that we can prove that
  \begin{equation}\label{eq:est.Ndof}
    N_{\mathrm{DOF}}(\wh{\lambda},F)^2
    \lesssim h_f^{1+d'-d}\norm{\partial f}{\wh{\theta}}^2 + h_f^{2+d'-d}\norm{f}{\wh{\xi}}^2\qquad
    \forall d'\in\{k,\ldots,d\}\,,\ \forall F\in\FM{d'}(\Sh(f)).
  \end{equation}
  Plugging this bound into \eqref{eq:est.lambda.Ndof} would then show that $\wh{\lambda}$ satisfies \eqref{eq:local.pb.estimate} with $\wh{\xi}$, $\wh{\theta}$ in the right-hand side, which would conclude the proof.

  It remains to show that \eqref{eq:est.Ndof} holds, which we do by induction on $d'$.
  If $d'=k$, the base case in the inductive process used to define the DOFs of $\wh{\lambda}$ show that $N_{\mathrm{DOF}}(\wh{\lambda},F)=0$ if $F\not\subset\partial f$ while, if $F\subset\partial f$,
  \begin{equation}\label{eq:est.Ndof.base}
    N_{\mathrm{DOF}}(\wh{\lambda},F)^2=N_{\mathrm{DOF}}(\wh{\theta},F)^2\le \norm{F}{\tr_F\wh{\theta}}^2\lesssim h_f^{d'-(d-1)}\norm{\partial f}{\wh{\theta}}^2,
  \end{equation}
  where the first inequality follows from the Cauchy--Schwarz inequality and the second from repeated discrete trace inequalities. This proves that \eqref{eq:est.Ndof} holds for $F$.

  Take now $d'\ge k+1$ and assume that \eqref{eq:est.Ndof} holds for all simplices of dimension $<d'$. Let $F\in\FM{d'}(\Sh(f))$.
  If $F\subset\partial f$, the inductive step in the construction of $\wh{\lambda}$ and the same arguments as above show that \eqref{eq:est.Ndof} holds for $F$. If $F\not\subset\partial f$ then the degrees of freedom of $\wh{\lambda}$ on $F$ are given by \eqref{eq:def.DOF.lambda}. Let $\mu\in \PL{s+k-d'}{d'-k}(F)$ and write $\mu=\DIFF\omega+\zeta$ according to the isomorphism \eqref{eq:decomposition.PL}.
  By \cite[Lemmas 16 and 17]{Di-Pietro.Droniou.ea:25}, we have $\norm{F}{\omega}\lesssim h_F \norm{F}{\mu}$. Applying Cauchy--Schwarz inequalities and the discrete trace inequality \cite[Lemma 18]{Di-Pietro.Droniou.ea:25} to the right-hand side of \eqref{eq:def.DOF.lambda} then shows that
  \begin{align}
    N_{\mathrm{DOF}}(\wh{\lambda},F)^2\lesssim{}& h_F^2\norm{F}{\tr_F\wh{\xi}}^2+h_F\norm{\partial F}{\tr_{\partial F}\wh{\lambda}}^2\nonumber\\
    \lesssim{}& h_F^{2+d'-d}\norm{f}{\wh{\xi}}^2+h_F\norm{\partial F}{\tr_{\partial F}\wh{\lambda}}^2.
    \label{eq:est.Ndof.induction.2}\end{align}
All the degrees of freedom of $\tr_{\partial F}\wh{\lambda}$ are on simplices of dimension $<d'$, for which \eqref{eq:est.Ndof} holds. Let $F'\in\FM{d'-1}(\Sh(F))$ and apply the bound \eqref{eq:est.lambda.Ndof} with $(f,\wh{\lambda})\leftarrow (F',\tr_{F'}\wh{\lambda})$ to get
\begin{align*}
    \norm{F'}{\tr_{F'}\wh{\lambda}}^2&\simeq \sum_{d''=k}^{d'-1} h_f^{d'-1-d''}\sum_{F''\in\FM{d''}(\Sh(F'))}N_{\mathrm{DOF}}(\tr_{F'}\wh{\lambda},F'')^2\\
\overset{\eqref{eq:est.Ndof}}&\lesssim \sum_{d''=k}^{d'-1} h_f^{d'-1-d''}\sum_{F''\in\FM{d''}(\Sh(F'))} 
\left(h_f^{1+d''-d}\norm{\partial f}{\wh{\theta}}^2+h_f^{2+d''-d}\norm{f}{\wh{\xi}}^2
\right)\\
&\lesssim h_f^{d'-d}\norm{\partial f}{\wh{\theta}}^2 + h_f^{1+d'-d}\norm{f}{\wh{\xi}}^2.
\end{align*}
Summing these estimates over $F'\in\FM{d'-1}(\Sh(F))$ and plugging the resulting inequality into \eqref{eq:est.Ndof.induction.2} shows that $N_{\mathrm{DOF}}(\wh{\lambda},F)$ satisfies \eqref{eq:est.Ndof}.
\end{proof}

\begin{remark}[Existence of a chain homotopy and comparison to classical constructions]
  In the proof of Lemma \ref{lem:first.local.problem} we have in fact produced an explicit (graded) operator
  $$
  \mathsf{H}^k : \PLtrim{s}{k} \longrightarrow \PLtrim{s+1}{k-1}
  $$
  such that, for every closed polynomial form $\zeta \in \PLtrim{s}{k}$ (i.e., $\DIFF \zeta = 0$),
  $$
  \bigl(\mathrm{Id} - W_{\Sh(f)}^{k}\circ R_{\Sh(f)}^{k}\bigr)(\zeta)
  = \DIFF\,\mathsf{H}^k(\zeta).
  $$
  One can extend the definition of $\mathsf{H}^k$ to all polynomial forms, yielding the full chain-homotopy identity on $\PLtrim{s}{k}$:
  $$
  \mathrm{Id} - W_{\Sh(f)}^{k}\circ R_{\Sh(f)}^{k}
  = \DIFF\circ \mathsf{H}^k + \mathsf{H}^{k+1}\circ \DIFF.
  $$
  While the existence of such a homotopy operator can be traced back to the de Rham--Whitney--Dodziuk theory \cite{Whitney:57,Dodziuk:76} and its later refinements in the language of high-order Whitney forms (e.g.\ \cite{Hiptmair:02}), our contribution lies in giving a fully explicit construction within the Finite Element Exterior Calculus framework of \cite{Arnold:18}, specialised to the trimmed polynomial spaces.
  This explicitness is essential for deriving the stability estimates and for ensuring that the constructed form is in trimmed polynomial spaces.
\end{remark}

For each $d$-simplex $F \in \FM{d}(\Sh(f))$, define the following subspace of forms whose trace vanishes on $\partial F$:
\[
\PLtrimz{s}{k}(F) \coloneq \{\omega\in \PLtrim{s}{k}(F)\,:\,\tr_{\partial F}\omega=0\}.
\]

\begin{proposition}[Local problem in bubble finite element space]
  \label{prop:petrov.galerkin}
  Let $d\ge 1$, $f\in\FM{d}(\Mh)$ be a $d$-cell, and $F\in\FM{d}(\Sh(f))$ be a $d$-simplex in $f$.
  Let $s\ge 1$ be a polynomial degree and $k\in\{0,\ldots,d\}$ be a form degree.
  Given $\zeta \in L^2\Lambda^{k}(F)$, there exists $\tau \in \PLtrimz{s+d-k+1}{k-1}(F)$ such that
  \begin{equation}
    (-1)^{k} \int_F \tau\wedge \DIFF \nu 
    =\int_F \zeta\wedge \nu  \qquad \forall \nu \in \KOSZUL_f \PL{s-1}{d-k+1}(F)
    \label{eq:petrov.galerkin}
  \end{equation}
  and
  \begin{equation}\label{eq:bubble.estimate}
    \norm{F}{\tau}\lesssim h_f\norm{F}{\zeta}.
  \end{equation}
\end{proposition}

\begin{proof}
  Let $\ell\in\{0,\ldots,d\}$ be a form degree and $t\ge 0$ be a polynomial degree. We start with a preliminary remark. The isomorphisms of \cite[Theorem 4.16]{Arnold.Falk.ea:06} or \cite[Corollary 3.3]{Berchenko-Kogan:21} show that, for all non-zero $\alpha\in\PLtrimz{t+\ell+1}{d-\ell}(F)$, there exists $\beta\in \PL{t}{\ell}(F)$ such that $B(\alpha,\beta)\coloneq \int_F \alpha\wedge \beta \not= 0$ (\cite[Theorem 4.16]{Arnold.Falk.ea:06} does not cover the case $\ell=d$, but this case is trivial since an isomorphism between $\PLtrimz{t+d+1}{0}(F)=\mathring{\mathcal{P}}_{t+d+1}\Lambda^0(F)$ and $\PL{t}{0}(F)$ is readily obtained by applying the star-Hodge operator and multiplying/factorising by the bubble function on $F$, which is polynomial of degree $d+1$).
  This proves the injectivity of the mapping
  \[
  \PLtrimz{t+\ell+1}{d-\ell}(F)\mapsto (\PL{t}{\ell}(F))'\,,\quad \alpha\to B(\alpha,\cdot).
  \]
  Since the domain and co-domains have the same dimension (use again \cite[Corollary 3.3]{Berchenko-Kogan:21} or \cite[Theorem 4.16]{Arnold.Falk.ea:06}), we infer that this mapping is an isomorphism. As a consequence, for any $L\in (\PL{t}{\ell}(F))'$, there exists $\alpha\in \PLtrimz{t+\ell+1}{d-\ell}(F)$ such that
  \begin{subequations}\label{eq:resolution.bubble.pb}
    \begin{align}
      \label{eq:resolution.bubble.pb.a}
      \int_F \alpha\wedge\beta = L(\beta)\qquad\forall \beta\in\PL{t}{\ell}(F),\\
      \label{eq:resolution.bubble.pb.b}
      \norm{F}{\alpha}\lesssim \sup_{\beta\in\PL{t}{\ell}(F)\backslash\{0\}}\frac{L(\beta)}{\norm{F}{\beta}}.
    \end{align}
  \end{subequations}
  Notice that the bound \eqref{eq:resolution.bubble.pb.b} follows from a reference simplex argument (to get a hidden constant that does not depend on $F$). 

  Let us now turn to the proof of the proposition, observing first that we can assume $k\ge 1$ (if $k=0$, then $\tau=0$ satisfies \eqref{eq:petrov.galerkin} since $\PL{s-1}{d-k+1}(F)=\{0\}$). Since $d-k+1\ge 1$, the mapping
  \begin{equation*} 
    \mathcal X_F\coloneq\KOSZUL_f\PL{s-1}{d-k+1}(F)\times \KOSZUL_f\PL{s-2}{d-k+2}(F)\mapsto\PL{s-1}{d-k+1}(F)\,,\quad (\nu,\mu)\to \DIFF\nu+\mu
  \end{equation*}
  is an isomorphism \cite[Eq.~(2.9)]{Bonaldi.Di-Pietro.ea:24}.
  The use of $\KOSZUL_f$ is made possible by the freedom in the choice of the base point in the Koszul operator, which can even be outside $F$ in its affine space.
  Using this decomposition, we can define $L:\PL{s-1}{d-k+1}(F)\to\Real$ by
  \[
  L(\DIFF\nu+\mu)\coloneq (-1)^k \int_F \zeta\wedge\nu.
  \]
  Consider then the solution $\alpha=\tau\in \PLtrimz{s+d-k+1}{k-1}(F)$ to \eqref{eq:resolution.bubble.pb} with $\ell=d-k+1\in\{0,\ldots,d\}$, $t= s-1$ and $L$ defined above. 
  Applying \eqref{eq:resolution.bubble.pb.a} to $\beta=\DIFF\nu$ with $\nu\in \KOSZUL_f\PL{s-1}{d-k+1}(F)$ shows that $\tau$ solves \eqref{eq:petrov.galerkin}.

  We now establish \eqref{eq:bubble.estimate}. Take $(\nu,\mu)\in\mathcal X_F$ and canonically extend them to $f$.
  By \cite[Lemmas 16 and 17]{Di-Pietro.Droniou.ea:25}, we have
  \begin{equation}\label{eq:bound.above}
    \norm{F}{\nu}
    \le \norm{f}{\nu}
    \lesssim h_f \norm{f}{\DIFF\nu}
    \lesssim h_f \norm{f}{\DIFF\nu+\mu}.
  \end{equation}
  By mesh regularity assumption, there is $x_F\in F$ and $\varrho\simeq 1$ such that, denoting by $B(x_F,a)$ the ball centered at $x_F$ and of radius $a$, it holds $B(x_F,\varrho h_f)\subset F\subset f\subset B(x_F,h_f)$. Since $\DIFF\nu+\mu$ is a polynomial form on $f$, an argument similar to the one in the proof of \cite[Lemma 1.25]{Di-Pietro.Droniou:20} then gives $\norm{f}{\DIFF\nu+\mu}\lesssim \norm{F}{\DIFF\nu+\mu}$. Combined with \eqref{eq:bound.above}, this yields $\norm{F}{\nu}\lesssim  h_f\norm{F}{\DIFF\nu+\mu}$. This estimate and a Cauchy--Schwarz inequality show that the right-hand side of \eqref{eq:resolution.bubble.pb.b} is $\lesssim h_f \norm{F}{\zeta}$, which proves \eqref{eq:bubble.estimate}.
\end{proof}

\subsection{Definition of the conforming lifting}\label{sec:def.lifting}

Let $k \in \{0, \dots, \dtop\}$ and $r\ge 0$. We start by defining local liftings 
\[
\lift{k}{f}: \uH{k}{f}\to \PLtrim{r+1+d-k}{k}(\Sh(f))\subset H\Lambda^{k}(f)
\]
on each $f \in \FM{d}(\Mh)$, through a recursive construction on the cell dimension $d\in \{k, \dots, \dtop\}$.

\begin{itemize}[leftmargin=1em]
\item \emph{Base case $d = k$.}  For any $f\in\FM{k}(\Mh)$ we set
  \begin{equation}
    \lift{k}{f} \ul{\omega}_f \coloneqq P^{k}_{r,f}\ul{\omega}_{f}\overset{\eqref{eq:def.P.d=k}}=\omega_f\in\PL{r}{k}(f)\subset \PLtrim{r+1}{k}(\Sh(f)).
    \label{eq:lift.base}
  \end{equation}

\item \emph{Induction step $d \in \{k+1, \dots, \dtop\}$.}  Let $f\in\FM{d}(\Mh)$ and assume, by induction, that $\lift{k}{f'} \ul{\omega}_{f'}\in\PLtrim{r+d-k}{k}(\Sh(f'))$ has already been constructed on all cells $f'\in\FM{d-1}(f)$.  
  To define $\lift{k}{f} \ul{\omega}_f\in\PLtrim{r+1+d-k}{k}(\Sh(f))$, we consider two cases based on the value of $d$.

  \begin{enumerate}[label=\arabic*),leftmargin=1em]
  \item \emph{Case $d=k+1$}.
    We set
    \begin{equation}
      \lift{k}{f} \ul{\omega}_f \coloneqq \lambda + \sigma,
      \label{eq:lift.sum}
    \end{equation}
    where the forms $\lambda$ and $\sigma$ are constructed in this order, as follows:

    \emph{1a) Construction of $\lambda$}:  
    We define $\lambda \in \PLtrim{r+1+d-k}{k}(\Sh(f))$ as the solution with minimal $L^2(f)$-norm of the following system on $f$:
    \begin{subequations} \label{eq:lift}
      \begin{empheq}[left=\empheqlbrace]{align}
        \DIFF \lambda &= \DIFF^{k}_{r,f} \ul{\omega}_f \label{eq:lift.a} \\
        \tr_{\partial f} \lambda &= \lift{k}{\partial f} \ul{\omega}_{\partial f}, \label{eq:lift.b}
      \end{empheq}
    \end{subequations}
    where $\lift{k}{\partial f} \ul{\omega}_{\partial f}$ is such that
    $\restr{(\lift{k}{\partial f} \ul{\omega}_{\partial f})}{f'} \coloneqq \lift{k}{f'} \ul{\omega}_{f'}$ for all $f'\in\FM{d-1}(\partial f)$. 

    \emph{1b) Construction of $\sigma$}:  
    We take 
    \[
    \sigma \in \PLtrimz{r+1+d-k}{k}(\Sh(f))\coloneq \left\{ \omega \in H\Lambda^{k}(f) \,:\, \restr{\omega}{F} \in \PLtrimz{r+1+d-k}{k}(F)\quad\forall F \in \FM{d}(\Sh(f))\right\}
    \]
    of the form
    \begin{equation}\label{eq:def.sigma}
      \sigma = \DIFF \tau,
    \end{equation}
    where $\tau \in \PLtrimz{r+2+d-k}{k-1}(\Sh(f))$ is the solution with minimal $L^2(f)$-norm of: for all $F \in \FM{d}(\Sh(f))$, 
    \begin{equation}
      (-1)^{k} \int_F \tau\wedge \DIFF \nu_F 
      =\int_F (P^{k}_{r,f} \ul{\omega}_f - \lambda)\wedge \nu_F  \qquad \forall \nu_F \in \KOSZUL_f \PL{r}{d-k+1}(F).
      \label{eq:tau.existence}
    \end{equation}

  \item \emph{Case $d \ge k+2$.}
    We first define the \emph{lifting correction} $\liftc{k}{f}: \uH{k}{f}\to  H\Lambda^{k+1}(f)$ as 
    \begin{equation}
      \liftc{k}{f} \ul{\omega}_f \coloneqq \chi + \psi\in \PLtrim{r+d-k}{k+1}(\Sh(f)),
      \label{eq:lift.correction.sum}
    \end{equation}
    where the terms $\chi$ and $\psi$ are constructed in this order, as follows:

    \emph{2a) Construction of $\chi$}:
    We take $\chi \in \PLtrim{r+d-k-1}{k+1}(\Sh(f))$ the solution with minimal $L^2(f)$-norm of the following system on $f$:
    \begin{subequations} \label{eq:lift.correction}
      \begin{empheq}[left=\empheqlbrace]{align}
        \DIFF \chi &= -\DIFF(\DIFF^{k}_{r,f} \ul{\omega}_f),  \label{eq:liftc.correction.a} \\
        \tr_{\partial f} \chi &= \DIFF \lift{k}{\partial f} \ul{\omega}_{\partial f} - \tr_{\partial f} \DIFF^{k}_{r,f} \ul{\omega}_f, \label{eq:lift.correction.b}
      \end{empheq}
    \end{subequations}
    where $\DIFF \lift{k}{\partial f} \ul{\omega}_{\partial f}$ is the broken exterior derivative such that
    $\restr{(\DIFF \lift{k}{\partial f} \ul{\omega}_{\partial f})}{f'} \coloneqq \DIFF \lift{k}{f'} \ul{\omega}_{f'}$ for all $f'\in\FM{d-1}(\partial f)$.

    \emph{2b) Construction of $\psi$}:  
    We take $\psi \in \PLtrimz{r+d-k}{k+1}(\Sh(f))$ of the form
    \begin{equation}
      \psi = \DIFF \rho,
      \label{eq:def.psi}
    \end{equation}
    where $\rho \in \PLtrimz{r+1+d-k}{k}(\Sh(f))$ is the solution with minimal $L^2(f)$-norm of: for all $F \in \FM{d}(\Sh(f))$,
    \begin{equation}
      (-1)^{k+1} \int_F \rho\wedge \DIFF \mu_F 
      =-\int_F  \chi\wedge \mu_F  \qquad \forall \mu_F \in \KOSZUL_f \PL{r}{d-k}(F).
      \label{eq:rho.existence}
    \end{equation}

    Then, to define $\lift{k}{f} \ul{\omega}_f\in\PLtrim{r+1+d-k}{k}(f)$, we perform exactly the same two–step construction as in the case $d=k+1$ (Steps 1a--1b), except that $\lambda\in\PLtrim{r+1+d-k}{k}(\Sh(f))$ is taken as the solution with minimal $L^2(f)$-norm of
    \begin{subequations} \label{eq:lift.d>=k+2}
      \begin{empheq}[left=\empheqlbrace]{align}
        \DIFF \lambda &=  \DIFF^{k}_{r,f} \ul{\omega}_f + \liftc{k}{f} \ul{\omega}_f \label{eq:lift.d>=k+2.a} \\
        \tr_{\partial f} \lambda &= \lift{k}{\partial f} \ul{\omega}_{\partial f}. \label{eq:lift.d>=k+2.b}
      \end{empheq}
    \end{subequations}
    The construction of $\sigma$ then remains the same as in Step 1b.
  \end{enumerate}

\end{itemize}

The global lifting $\lift{k}{h}:\uH{k}{h}\to H\Lambda^{k}(\Omega)$ is then defined by setting
\begin{equation}\label{eq:def.lifting}
  (\lift{k}{h}\ul{\omega}_h)|_f \coloneqq
  \lift{k}{f}\ul{\omega}_f\qquad\forall\ul{\omega}_h\in\uH{k}{h}\,,\quad\forall f\in\FM{\dtop}(\Mh).
\end{equation}

\begin{lemma}[Well-posedness of the lifting construction]
  Each of the problems \eqref{eq:lift}, \eqref{eq:tau.existence}, \eqref{eq:lift.correction}, and \eqref{eq:rho.existence} admits at least one solution,
  and the equation \eqref{eq:def.lifting} defines a linear lifting  $\lift{k}{h}:\uH{k}{h}\to H\Lambda^{k}(\Omega)$.
\end{lemma}

\begin{proof}
  We first prove that each problem in the recursive construction has a solution in the considered space.

  \smallskip\noindent
  \emph{i) Construction of $\lambda$, case $d=k+1$}.
  Applying the definition \eqref{eq:def.d} with $\mu=1\in\PL{r}{0}(f)$ yields
  \[
  \int_f \DIFF^k_{r,f}\ul{\omega}_f = \int_{\partial f} P^k_{r,\pf} \ul{\omega}_f\overset{\eqref{eq:lift.base}}=\int_{\partial f} \lift{k}{\pf} \ul{\omega}_{\pf}.
  \]
  Hence, $(\xi,\theta)=(\DIFF^k_{r,f}\ul{\omega}_f,\lift{k}{\pf} \ul{\omega}_{\pf})$ satisfy the compatibility condition \eqref{eq:condition.d=k+1} and the existence of a solution $\lambda\in\PLtrim{r+1+d-k}{k}(\Sh(f))$ to \eqref{eq:lift} is a straightforward consequence of Lemma \ref{lem:first.local.problem}
  (with $s=r+1$) after noticing that $\DIFF^{k}_{r,f} \ul{\omega}_f\in\PL{r}{k+1}(f)\subset\PLtrim{r+1}{k+1}(\Sh(f))$
  and $\lift{k}{\partial f} \ul{\omega}_{\partial f}\in \PLtrim{r+1}{k}(\Sh(\partial f))$ by \eqref{eq:lift.base}.

  \smallskip\noindent
  \emph{ii) Construction of $\sigma$}.
  By Proposition \ref{prop:petrov.galerkin} with $s=r+1$, for each $F\in\FM{d}(\Sh(f))$ the Petrov--Galerkin problem \eqref{eq:tau.existence} admits a solution in $\PLtrimz{r+2+d-k}{k-1}(F)$. Since the traces of these solutions vanish on $\partial F$, we can glue them together to obtain $\tau\in\PLtrimz{r+2+d-k}{k}(\Sh(f))$. The exterior derivative commutes with the trace and preserves trimmed spaces, so $\sigma=\DIFF\tau$ belongs to $\PLtrimz{r+1+d-k}{k}(\Sh(f))$.

  \smallskip\noindent
  \emph{iii) Construction of $\chi$}.
  This construction is only relevant for $d\ge k+2$, and thus 
  \[
  \DIFF(\DIFF^k_{r,f}\ul{\omega}_f)\in\PL{r-1}{k+2}(f)\subset \PLtrim{r+d-k-1}{k+2}(\Sh(f))
  \]
  and, by induction hypothesis on $\lift{k}{\partial f}$,
  \[
  \DIFF \lift{k}{\partial f} \ul{\omega}_{\partial f} - \tr_{\partial f} \DIFF^{k}_{r,f} \ul{\omega}_f\in\PLtrim{r+d-1-k}{k+1}(\Sh(\partial f))+
  \PLtrim{r}{k+1}(\Sh(\partial f))\subset \PLtrim{r+d-k-1}{k+1}(\Sh(\partial f)).
  \]
  The existence of a solution $\chi\in \PLtrim{r+d-k}{k+1}(\Sh(f))$ to \eqref{eq:lift.correction} is then obtained by applying Lemma \ref{lem:first.local.problem} (with $s=r+d-k-1\ge 1$ and $k\leftarrow k+1$), provided we check the compatibility condition  \eqref{eq:local.pb.cond}.
  If $d=k+2=(k+1)+1$ then we need to check \eqref{eq:condition.d=k+1}, which amounts to
  \[
  \int_f -\DIFF(\DIFF^k_{r,f}\ul{\omega}_f)= \int_{\pf}\left(\DIFF \lift{k}{\partial f} \ul{\omega}_{\partial f} - \tr_{\partial f} \DIFF^{k}_{r,f} \ul{\omega}_f\right).
  \] 
This relation is a straightforward consequence of the Stokes formula on $f$ for $\DIFF(\DIFF^k_{r,f}\ul{\omega}_f)$ and on $\pf$ for $\DIFF \lift{k}{\partial f} \ul{\omega}_{\partial f}$ (noticing that $\partial\pf=\emptyset$).
If $d>k+2=(k+1)+1$, we need to verify \eqref{eq:condition.d>=k+2}, which directly follows from the commutation of exterior derivative and trace:
  \[
  \tr_{\partial f}(-\DIFF(\DIFF^k_{r,f}\ul{\omega}_f)) = -\DIFF (\tr_{\partial f}\DIFF^k_{r,f}\ul{\omega}_f)
  \overset{\DIFF^2=0}=\DIFF\left(\DIFF \lift{k}{\partial f} \ul{\omega}_{\partial f} - \tr_{\partial f} \DIFF^{k}_{r,f} \ul{\omega}_f\right).
  \]
  \smallskip\noindent
  \emph{iv) Construction of $\psi$}.
  The argument is the same as for $\sigma$, except that we apply Proposition \ref{prop:petrov.galerkin} with $s=r+1$ and $k\leftarrow k+1$, since we look for a $k$-form $\rho$ and \eqref{eq:rho.existence} is required to hold for all $\nu_F\in\PLtrim{r}{d-(k+1)+1}(F)$. As a consequence, $\rho\in\PLtrimz{r+1+d-k}{k}(\Sh(f))$ and $\psi=\DIFF\rho\in\PLtrim{r+d-k}{k+1}(\Sh(f))$.

  \smallskip\noindent
  \emph{v) Construction of $\lambda$, case $d\ge k+2$}. 
  We still aim at applying Lemma \ref{lem:first.local.problem}. The problem \eqref{eq:lift.d>=k+2} defining $\lambda$ has the form \eqref{eq:local.pb} with data
  \[
  \xi\coloneq \DIFF^{k}_{r,f} \ul{\omega}_f + \liftc{k}{f} \ul{\omega}_f\in\PLtrim{r+d-k}{k+1}(\Sh(f))\quad\text{ and }\quad
  \theta\coloneq \lift{k}{\partial f} \ul{\omega}_{\partial f}\in\PLtrim{r+d-k}{k}(\Sh(\pf)),
  \]
  so the existence of a solution $\lambda\in \PLtrim{r+1+d-k}{k}(\Sh(f))$ to \eqref{eq:lift.d>=k+2} is ensured if we check the compatibility condition \eqref{eq:condition.d>=k+2}. Since $\tr_{\pf}\psi=0$ we have 
  \[
  \tr_{\pf}\liftc{k}{f}\ul{\omega}_f=\tr_{\pf}\chi\overset{\eqref{eq:lift.correction.b}}=\DIFF \lift{k}{\partial f} \ul{\omega}_{\partial f} - \tr_{\partial f} \DIFF^{k}_{r,f} \ul{\omega}_f.
  \]
  Hence, $\tr_{\pf}\xi=\DIFF \lift{k}{\partial f} \ul{\omega}_{\partial f}=\DIFF\theta$ as required.

  \medskip

  This proves that all local problems involved in the construction of $\lift{k}{h}$ have at least one solution. Choosing the solutions with minimal $L^2(f)$-norms additionally enforces an orthogonality condition with respect to the kernel of these local problems, and ensures that each of these (extended) problems is uniquely solvable. Hence, the selected solutions are linear with respect to the data, and thus with respect to $\ul{\omega}_h$. This proves that $\lift{k}{h}$ itself is linear.

  For all $f\in\FM{\dtop}(\Mh)$, we have $\tr_{\pf}\sigma=0$ and thus the equations \eqref{eq:lift.b} and \eqref{eq:lift.d>=k+2.b} ensure that, on each $f'\in\FM{d-1}(\Mh)$, we have $\tr_{f'}\lift{k}{f}\ul{\omega}_f=\lift{k}{f'}\ul{\omega}_{f'}$. This quantity only depends on $f'$, not on $f$, and thus if $f'$ is an interface between two $d$-cells $f_1$ and $f_2$ we have $\tr_{f'}\lift{k}{f_1}\ul{\omega}_{f_1}=\tr_{f'}\lift{k}{f_2}\ul{\omega}_{f_2}$. Since each local lifting belongs to $H\Lambda^k$ on its cell, and since the traces of local liftings match between adjacent $d$-cells, this proves that $\lift{k}{h}\ul{\omega}_h\in H\Lambda^k(\Omega)$. \end{proof}

\begin{remark}[Structure of the lifting]\label{rem:conclusion.proof}
Extending the argument applied in the conclusion of the proof above, we see that $\tr_f(\lift{k}{h}\ul{\omega}_h)=\lift{k}{f}\ul{\omega}_f$ for any cell $f$ of any dimension $d\ge k$, thus showing the single-valuedness and polynomial characteristic of the traces of the lifting, as stated in Remark \ref{rem:piecewise.lifting}.

If $\ul{\omega}_f=0$ for some $f\in\FM{d}(\Mh)$, then the linearity of the local lifting shows that $\lift{k}{f}\ul{\omega}_f=0$, and thus that $\tr_f(\lift{k}{h}\ul{\omega}_h)=0$ as claimed in Remark \ref{rem:homogeneous.BC}.
\end{remark}

\subsection{Properties of the lifting}\label{sec:properties.lifting}

In this section we prove that the lifting satisfies the projection relation \eqref{eq:lift.proj} and the estimates \eqref{eq:lift.bound.0}--\eqref{eq:lift.bound.1}.

\medskip

\underline{Proof of \eqref{eq:lift.proj}.}

As observed in Remark \ref{rem:conclusion.proof}, for all $d\in\{k,\ldots,\dtop\}$ and all $f\in\FM{d}(\Mh)$ we have $\tr_f(\lift{k}{h}\ul{\omega}_h)=\lift{k}{f}\ul{\omega}_f$. Hence, \eqref{eq:lift.proj} is equivalent to the same relation with the trace replaced by the lifting on the cell $f$.

The proof of this relation is done by induction on the dimension $d$ of $f$.
In the case $d=k$, the definition of $\lift{k}{f} \ul{\omega}_f$ in \eqref{eq:lift.base} immediately yields \eqref{eq:lift.proj}.
We now take $d \in \{k+1, \dots, \dtop\}$, and assume that \eqref{eq:lift.proj} holds on every $(d-1)$-cell.
Let us first recall the following decomposition
\[
\PLtrim{r+1}{d-k}(f) = \DIFF \KOSZUL_f\PL{r}{d-k}(f) \oplus \KOSZUL_f \PL{r}{d-k+1}(f),
\]
which follows from \eqref{eq:def.Ptrim} and \eqref{eq:decomp.PL} (with $r\leftarrow r+1$ and $k\leftarrow d-k\ge 1$), and $\DIFF^2=0$.
Owing to this decomposition, we only need to separately prove \eqref{eq:lift.proj} for $\zeta\in\DIFF\KOSZUL_f\PL{r}{d-k}(f)$ (which we split in two cases) and for $\zeta\in \KOSZUL_f\PL{r}{d-k+1}(f)$.

\smallskip
\emph{i) Case $\zeta\in \DIFF\KOSZUL_f\PL{r}{d-k}(f)$ and $d=k+1$}.
We write $\zeta=\DIFF\mu$ with $\mu\in\KOSZUL_f\PL{r}{d-k}(f)$.
By \eqref{eq:lift.sum}, the fact that $\DIFF \sigma \overset{\eqref{eq:def.sigma}}= \DIFF^2\tau=0$, and \eqref{eq:lift.a}, we have
$\DIFF \lift{k}{f} \ul{\omega}_f = \DIFF^{k}_{r,f} \ul{\omega}_f$. Hence, testing \eqref{eq:def.pot} against $(\mu,0)$, we find
\begin{align}
  \label{eq:passage}
  \begin{split}
    (-1)^{k+1} \int_f  P^{k}_{r,f} \ul{\omega}_f \wedge \DIFF \mu 
    &= \int_f \DIFF^{k}_{r,f} \ul{\omega}_f\wedge \mu  - \int_\pf P^{k}_{r,\pf} \ul{\omega}_\pf\wedge \tr_\pf \mu \\
    &=
    \int_f \DIFF \lift{k}{f} \ul{\omega}_f \wedge \mu 
    -\int_\pf \tr_{\partial f} \lift{k}{f} \ul{\omega}_f \wedge \tr_\pf \mu\\
    \overset{\text{Stokes}}&= (-1)^{k+1} \int_f \lift{k}{f} \ul{\omega}_f\wedge \DIFF \mu,
  \end{split}
\end{align}
where, to handle the boundary term in the second equality, we have used the induction hypothesis on each $f' \in \FM{d-1}(\partial f)$, after noticing that $\tr_{f'}\lift{k}{f}\ul{\omega}_f=\lift{k}{f'}\ul{\omega}_{f'}$ and $\tr_{f'} \mu \in \PLtrim{r+1}{(d-1)-k}(f')$ (since $\mu\in\PLtrim{r+1}{d-k-1}(f)$).

\smallskip
\emph{ii) Case $\zeta\in \DIFF\KOSZUL_f\PL{r}{d-k}(f)$ and $d\ge k+2$}.
We still write $\zeta=\DIFF\mu$ with $\mu\in\KOSZUL_f\PL{r}{d-k}(f)$, and consider first the correction term $\liftc{k}{f} \ul{\omega}_f$.
By \eqref{eq:lift.correction.sum} and \eqref{eq:def.psi} we have $\liftc{k}{f} \ul{\omega}_f=\chi+\DIFF\rho$. For any simplex $F\in\FM{d}(\Sh(f))$
we have $\tr_{\partial F}\rho=0$ since $\rho\in\PLtrimz{r+1+d-k}{k}(\Sh(f))$, and we can therefore use the Stokes formula to write
\begin{align*}
  \int_F \liftc{k}{f} \ul{\omega}_f \wedge \mu 
  &= \int_F \chi \wedge \mu  + \int_F \DIFF\rho \wedge \mu \\
  &= \int_F \chi \wedge \mu + (-1)^{k+1} \int_F \rho \wedge \DIFF \mu \\
  \overset{\eqref{eq:rho.existence}}&= \int_F\chi\wedge \mu -\int_F  \chi\wedge \mu = 0.
\end{align*}
Notice that, in the last equality, we have used $\mu|_F\in\KOSZUL_f\PL{r}{d-k}(F)$ (this is why \eqref{eq:rho.existence} had to be considered with the Koszul operator on $f$, not on $F$). Summing these relations over $F\in\FM{d}(\Sh(f))$ and recalling that $\DIFF \lift{k}{f} \ul{\omega}_f=\DIFF\lambda\overset{\eqref{eq:lift.d>=k+2.a}}=\DIFF^{k}_{r,f} \ul{\omega}_f+\liftc{k}{f} \ul{\omega}_f $, we infer that
\[
\int_f \DIFF \lift{k}{f} \ul{\omega}_f \wedge \mu = \int_f \DIFF^{k}_{r,f} \ul{\omega}_f\wedge \mu.
\]
The second equality in \eqref{eq:passage} is therefore still valid, and the argument used in Case (i) above can be applied to conclude.

\smallskip
\emph{iii) Case $\zeta\in \KOSZUL_f\PL{r}{d-k+1}(f)$}. Take $\lambda\in \PLtrim{r+1+d-k}{k}(\Sh(f))$ and $\tau\in\PLtrimz{r+2+d-k}{k-1}(\Sh(f))$ such that $\lift{k}{f} \ul{\omega}_f=\lambda+\DIFF\tau$ and write
\[
\int_f P^{k}_{r,f} \ul{\omega}_f\wedge \zeta 
= \int_f \lambda\wedge \zeta + \int_f (P^{k}_{r,f} \ul{\omega}_f - \lambda)\wedge \zeta 
\overset{\eqref{eq:tau.existence}}= \int_f \lambda\wedge \zeta 
+ (-1)^{k} \int_f \tau\wedge \DIFF \zeta,
\]
where the use of \eqref{eq:tau.existence} is justified since $\zeta|_F\in\KOSZUL_f\PL{r}{d-k+1}(F)$ for all $F\in\FM{d}(\Sh(f))$. 
Since $\tr_{\partial f}\tau=0$ we can apply the Stokes formula to conclude the proof of \eqref{eq:lift.proj}:
\[
\int_f P^{k}_{r,f} \ul{\omega}_f\wedge \zeta 
= \int_f \lambda\wedge \zeta  + \int_f  \DIFF \tau\wedge \zeta 
= \int_f \lift{k}{f} \ul{\omega}_f\wedge \zeta.
\]

\medskip

\underline{Proof of \eqref{eq:dlift.proj}.}

Assume that $f\in\FM{d}(\Mh)$ with $d\ge k+1$ and let $\mu\in\PLtrim{r+1}{d-k-1}(f)$. The Stokes formula gives
\begin{equation}\label{eq:relabove}
   \int_f \DIFF\tr_f(\lift{k}{h} \ul{\omega}_h) \wedge \mu=(-1)^{k+1}\int_f \tr_f(\lift{k}{h} \ul{\omega}_h)\wedge\DIFF\mu + \int_{\pf}\tr_{\pf}(\lift{k}{h} \ul{\omega}_h)\wedge \tr_{\pf}\mu.
\end{equation}
We note that $\DIFF\mu\in\PL{r}{d-k}(f)\subset\PLtrim{r+1}{d-k}(f)$ so, by \eqref{eq:lift.proj}, we can substitute $P^k_{r,f}\ul{\omega}_f$ for $\tr_f(\lift{k}{h} \ul{\omega}_h)$ in the right-hand side of \eqref{eq:relabove}.
By \cite[Lemma 4]{Bonaldi.Di-Pietro.ea:24}, we also have $\tr_{f'}\mu\in\PLtrim{r+1}{d-k-1}(f')$ for all $f'\in\FM{d-1}(f)$, and we can thus apply \eqref{eq:lift.proj} to each $f'\in\FM{d-1}(f)$ instead of $f$ in order to also substitute $P^k_{r,\pf} \ul{\omega}_{\pf}$ for $\tr_{\pf}(\lift{k}{h} \ul{\omega}_h)$ in \eqref{eq:relabove}. This gives
\[
   \int_f \DIFF\tr_f(\lift{k}{h} \ul{\omega}_h) \wedge \mu=(-1)^{k+1}\int_f P^k_{r,f}\ul{\omega}_f \wedge\DIFF\mu + \int_{\pf}P^k_{r,\pf} \ul{\omega}_{\pf}\wedge \tr_{\pf}\mu.
\] 
Invoking then \eqref{eq:ipp.pot} concludes the proof of \eqref{eq:dlift.proj}.

\medskip

\underline{Proof of \eqref{eq:lift.bound.0}--\eqref{eq:lift.bound.1}.}

We first recall the following estimates from \cite[Lemma 8]{ Di-Pietro.Droniou.ea:24}, valid for all $f\in\FM{d}(\Mh)$ (with $d\ge k$ for the first one and $d\ge k+1$ for the second) and all $\ul{\omega}_f\in\uH{k}{f}$ :
\begin{gather}\label{eq:bound.pot}
  \norm{f}{P^k_{r,f}\ul{\omega}_f}\lesssim \opn{f}{\ul{\omega}_f},\\
  \label{eq:bound.domega}
  \norm{f}{\DIFF^k_{r,f}\ul{\omega}_f}\le \opn{f}{\ul{\DIFF}^k_{r,f}\ul{\omega}_f}\lesssim h_f^{-1}\opn{f}{\ul{\omega}_f}.
\end{gather}
Since, in the construction of $\lift{k}{f}\ul{\omega}_f$, we have selected minimal-norm solutions to local problems, we also note that the estimates from Lemma \ref{lem:first.local.problem} and Proposition \ref{prop:petrov.galerkin} are applicable to these solutions. We will now establish the bounds \eqref{eq:lift.bound.0}--\eqref{eq:lift.bound.1} in a recursive way on the dimension $d$ of $f$, following the construction of the lifting.

\smallskip
\emph{i) Case $d=k$}. The definition \eqref{eq:lift.base} of the lifting together with \eqref{eq:bound.pot} directly gives \eqref{eq:lift.bound.0}.

\smallskip
\emph{ii) Case $d=k+1$}. Recall that $\tr_f(\lift{k}{h}\ul{\omega}_h)=\lift{k}{f}\ul{\omega}_f=\lambda+\sigma$ is given by \eqref{eq:lift.sum}.
As $\lambda$ solves \eqref{eq:lift}, Lemma \ref{lem:first.local.problem} shows that
\[
\norm{f}{\lambda}^2
\lesssim h_f \norm{\pf}{\lift{k}{\pf}\ul{\omega}_{\pf}}^2
+ h_f^2\norm{f}{\DIFF^k_{r,f}\ul{\omega}_f}^2.
\]
Invoking the bound on $\lift{k}{f'}\ul{\omega}_{f'}$, for $f'\in\FM{d-1}(f)$, coming from the induction assumption and \eqref{eq:bound.domega} gives
\begin{equation}\label{eq:estim.lift.lambda}
  \norm{f}{\lambda}^2
  \lesssim h_f \sum_{f'\in\FM{d-1}(f)}\opn{f'}{\ul{\omega}_{f'}}^2 + \opn{f}{\ul{\omega}_f}^2
  \overset{\eqref{eq:opn.f.k+}}\le \opn{f}{\ul{\omega}_f}^2.
\end{equation}
Regarding $\sigma=\DIFF\tau$ with $\tau$ satisfying \eqref{eq:tau.existence}, the inverse inequality of \cite[Lemma 15]{Di-Pietro.Droniou.ea:25} applied in each $F\in\FM{d}(\Sh(f))$ and Proposition \ref{prop:petrov.galerkin} yield
\[
\norm{f}{\sigma}^2
\lesssim h_f^{-2} \norm{f}{\tau}^2
\lesssim \norm{f}{P^k_{r,f}\ul{\omega}_f-\lambda}^2.
\]
Invoking \eqref{eq:bound.pot} and \eqref{eq:estim.lift.lambda} then show that 
\begin{equation}\label{eq:est.lift.sigma}
  \norm{f}{\sigma}^2\lesssim \opn{f}{\ul{\omega}_f}^2,
\end{equation}
which, combined with \eqref{eq:estim.lift.lambda}, proves \eqref{eq:lift.bound.0}.

The bound \eqref{eq:lift.bound.1} directly follows from $\DIFF\lift{k}{f}\ul{\omega}_f=\DIFF(\lambda+\DIFF\tau)=\DIFF\lambda=\DIFF^k_{r,f}\ul{\omega}_f$ and from \eqref{eq:bound.domega}.

\smallskip
\emph{iii) Case $d\ge k+2$}. We start by estimating the correction $\liftc{k}{f}\ul{\omega}_f=\chi+\psi$ defined by \eqref{eq:lift.correction.sum}.
Since $\chi$ solves \eqref{eq:lift.correction}, Lemma \ref{lem:first.local.problem} gives
\begin{align}
  \norm{f}{\chi}^2
  &\lesssim h_f\norm{\pf}{\DIFF\lift{k}{\pf}\ul{\omega}_{\pf} - \tr_{\pf}\DIFF^k_{r,f}\ul{\omega}_f}^2
  + h_f^2 \norm{f}{\DIFF(\DIFF^k_{r,f}\ul{\omega}_f)}^2\nonumber\\
  &\lesssim h_f\sum_{f'\in\FM{d-1}(f)}\opn{f'}{\ul{\DIFF}^k_{r,f'}\ul{\omega}_{f'}}^2+\norm{f}{\DIFF^k_{r,f}\ul{\omega}_f}^2\nonumber\\
  \overset{\eqref{eq:opn.f.k+}}&= \opn{f}{\ul{\DIFF}^k_{r,f}\ul{\omega}_f}^2,
\label{eq:est.lift.chi}
\end{align}
where we have used the bound on $\DIFF\lift{k}{f'}\ul{\omega}_{f'}$, for $f'\in\FM{d-1}(f)$, coming from the induction hypothesis together discrete trace and inverse inequalities \cite[Lemmas 15 and 18]{Di-Pietro.Droniou.ea:25} in the second inequality.
The bound on $\psi$ given by \eqref{eq:def.psi}--\eqref{eq:rho.existence} follows from inverse inequalities and Proposition \ref{prop:petrov.galerkin}:
\[
\norm{f}{\psi}^2
\lesssim h_f^{-2} \norm{f}{\rho}^2
\lesssim \norm{f}{\chi}^2\overset{\eqref{eq:est.lift.chi}}\lesssim
\opn{f}{\ul{\DIFF}^k_{r,f}\ul{\omega}_f}^2.
\]
Combining this with \eqref{eq:est.lift.chi} shows that 
\begin{equation}\label{eq:est.lift.corr}
  \norm{f}{\liftc{k}{f}\ul{\omega}_f}^2
  \lesssim \opn{f}{\ul{\DIFF}^k_{r,f}\ul{\omega}_f}^2
\end{equation}
and thus, by Lemma \ref{lem:first.local.problem}, that $\lambda$ solution to \eqref{eq:lift.d>=k+2} satisfies
\[
\norm{f}{\lambda}^2
\lesssim h_f \norm{\pf}{\lift{k}{\pf}\ul{\omega}_{\pf}}^2
+ h_f^2 \norm{f}{\DIFF^k_{r,f}\ul{\omega}_f}^2
+ h_f^2 \opn{f}{\ul{\DIFF}^k_{r,f}\ul{\omega}_f}^2.
\]
Applying \eqref{eq:bound.domega} and the bound on $\lift{k}{f'}\ul{\omega}_{f'}$, for $f'\in\FM{d-1}(f)$, coming from the induction hypothesis yields $\norm{f}{\lambda}^2\lesssim \opn{f}{\ul{\omega}_f}^2$. Combined with \eqref{eq:est.lift.sigma} (which is still valid for $d\ge k+2$ by the same arguments as in the case $d=k+1$), this establishes \eqref{eq:lift.bound.0}.

The bound \eqref{eq:lift.bound.1} directly follows from $\DIFF\lift{k}{f}\ul{\omega}_f=\DIFF\lambda=\DIFF^k_{r,f}\ul{\omega}_f+\liftc{k}{f}\ul{\omega}_f$ and \eqref{eq:est.lift.corr}.


\section*{Acknowledgements}
Funded by the European Union (ERC Synergy, NEMESIS, project number 101115663).
Silvano Pitassi also acknowledges the funding of the European Union via the MSCA EffECT, project number 101146324.
Views and opinions expressed are however those of the authors only and do not necessarily reflect those of the European Union or the European Research Council Executive Agency. Neither the European Union nor the granting authority can be held responsible for them.


\printbibliography

\end{document}